\documentclass[12 pt]{article}
\usepackage{graphicx} 
\usepackage{times}\usepackage{setspace}
\usepackage{graphicx}
\usepackage{float}
\usepackage{caption}
\usepackage[caption = false]{subfig}
\usepackage{hyperref}
\usepackage{natbib}
\usepackage{amssymb}
\usepackage{amsthm}
\newtheorem{dfn}{Definition}

\usepackage[english]{babel}

\newtheorem{theorem}{Theorem}
\newtheorem{lemma}[theorem]{Lemma}
\usepackage{enumerate}
\usepackage{nccmath}
\usepackage{setspace}
\usepackage{amsmath}
\usepackage{blindtext, multicol}
\usepackage{multicol}
\usepackage{mathrsfs}
\usepackage[autostyle, english = american]{csquotes}
\MakeOuterQuote{"}
\newtheorem{Lemma}{LEMMA}[section]
\newtheorem{Proposition}[Lemma]{Proposition}
\def\bp{\begin{Proposition}}
\def\ep{\end{Proposition}}
\newtheorem{Corollary}[theorem]{Corollary}

\allowdisplaybreaks
\makeatletter
\renewcommand*{\@fnsymbol}[1]{\@arabic{#1}}
\makeatother

	\newcommand{\n}[1]{\lVert#1\rVert} 
	\newcommand{\p}[1]{{#1}_{\rho}} 
	\providecommand{\keywords}[1]
	{
\textbf{Keywords: } #1
}

\title{\Large Regularization of Statistical Inverse Problems on Non-Reflexive Banach Spaces  }
\author{\small Darrel K Joseph\thanks{The author was supported through the Junior Research Fellowship by the University Grants Commission (UGC),
		India-NTA Ref. no.: 221610060111.\\Indian Institute of Science Education and Research, Thiruvananthapuram, India. \texttt{darreljoseph23@iisertvm.ac.in}}
	$~$ and $~$ M P Rajan\thanks{Indian Institute of Science Education and Research, Thiruvananthapuram, India. \texttt{rajanmp@iisertvm.ac.in}.}}
\date{}
\begin{document}
	
	\maketitle
	\noindent\textbf{Abstract} Inverse learning within a statistical framework has a wide range of applications. It has garnered significant attention in machine learning, artificial intelligence, and related fields, where the goal is to infer unknown parameters from indirect and noisy observations. This work investigates the stable approximation of $u^{\dagger}$ which solves the equation $Au=g$, with $A$ being a linear operator between appropriate vector spaces. We will consider the domain to be a non-reflexive Banach Space and the co-domain to be a space of real-valued functions on a metric space $X$. The  function $g$ is characterized by a finite number of independently and identically distributed data points, which are assumed to follow some unknown probability measure $\rho$. We employ Tikhonov regularization with an arbitrary convex functional to obtain the regularized solution corresponding to the given data point. The convergence analysis is carried out with respect to the Bregman distance, and an upper bound for the error is derived in probability terms. The theoretical findings are then supported by numerical experiments.

	\noindent \keywords{Reproducing kernel Banach space $\cdot$ Inverse problem $\cdot$ Statistical learning $\cdot$ Convex analysis $\cdot$ Regularization}

	\section{Introduction}
	
    In the literature of inverse problems, we are usually concerned with the equations of the form $Ax=y$, where $A:B_1\to B_2$ is a bounded linear operator between appropriate Banach spaces. Given a fixed element $y\in B_2$, we are required to find $x^{\dagger}\in B_1$ such that $Ax^{\dagger}=y$. In many practical scenarios, the exact element $y$ is unknown and only a perturbed data $y^{\delta}$ is known, with $\|y^{\delta}-y\|_{B_2}\leq\delta$, for some $\delta>0$. Hence, we seek the solution $x^{\delta}$ of the equation $Ax=y^{\delta}$. 
    
    Naturally, we want the perturbed solution $x^{\delta}$ to converge to $x^{\dagger}$ as $\delta\to0$. However, if our operator $A$ is ill-posed, that is, does not have a continuous inverse, then the solution $x^{\delta}$ may fail to exist, or may not converge to $x^{\dagger}$ at all. In such cases, we apply different regularization strategies to obtain a regularized solution. In simple terms, we replace the ill-posed equation with a well-posed one to obtain a regularized solution. Let $x_{\lambda}^{\delta}$ be the regularized solution corresponding to $y^{\delta}$, with $\lambda$ being a regularization parameter. We then usually choose $\lambda$ depending on $\delta$ such that
    \begin{equation*}
    	\|x_{\lambda}^{\delta}-x^{\dagger}\|_{B_1}\to 0~~~~\text{as}~~~~\delta,\lambda(\delta)\to0
    \end{equation*}
    The above problem is for deterministic cases. This has been well studied during the late twentieth and early twenty-first century, and
    the prominent work can be seen in
    \cite{engl2000regularization,nair2009linear,schuster2012regularization}.\\
    
    Statistical inverse problems, particularly inverse learning, are of growing interest in this field. Unlike the deterministic case, the error is modeled as a random variable that follows some unknown probability distribution. This ensures that the regularized solution will also take the form of a random variable and must follow a certain distribution. As a result, the convergence and the convergence rate are obtained in terms of expectations or probability. This probabilistic and principled framework for
    integrating data and prior knowledge makes statistical inverse
    problems particularly powerful in modern, data-driven applications
    where noise and model uncertainty are significant concerns.
    Several recent works have addressed these challenges (see
    \cite{BGM, BAUER200752, Bubba2025O, Learningrate}).\\
    
    In statistical inverse problems, we often consider a linear operator $A$ from a real Banach space $B_1$ to a vector space $V$ of real-valued functions on a metric space $X$. Let $z_i=\{x_i,y_i\},1\leq i\leq m$, be a set of identically and independently distributed (i.i.d.) data points in $(X\times [-M,M])$, where $M>0$, sampled from an unknown distribution $\rho(x,y)$. The aim is to approximate the element $u_{\rho}\in B_1$ such that $u_{\rho}$ is the minimizer of the risk functional
    \begin{equation*}
    	\int_{X\times[-M,M]}|y-(Au)(x)|^p ~d\rho,~p>1
    \end{equation*}
    over all $u\in B_1$. When $B_1$ is a Hilbert space, an optimal convergence rate is obtained in \cite{BGM} for $p=2$. In \cite{Learningrate}, the study is carried out for $p=2$ in a Banach space setting for the special case $A=I$, which is called the direct problem. In our previous work \cite{JOSEPH2026102050}, we extended this framework to the inverse problem setting by introducing a more general operator $A$, and hence providing a unified approach to both the direct and inverse problems. But in both \cite{Learningrate} and \cite{JOSEPH2026102050}, the domain space $B_1$ considered is a $q$-uniform convex Banach space \eqref{qcon}. Since uniformly convex spaces are reflexive, the obtained convergence rates are not applicable to non-reflexive spaces like the $l^1$, the space of absolutely summable real sequences. In this paper, we extend the existing framework to the case where the underlying domain is an arbitrary Banach space, thereby providing a more general and widely applicable approach. This is reinforced by utilizing a more general convex function as the penalty function instead of the previous norm power functional. This allows us to treat inverse problems in non-reflexive spaces that arise in sparse reconstruction, signal processing, and other applications. Consequently, the given framework will significantly expand the range of applicability of convergence rate results. The convergence and the convergence rates are given in terms of the Bregman distance \ref{breg}. Since the regularization functional is generally nonsmooth, and as the underlying space does not possess a Hilbert structure, norm estimates are often not the appropriate error measure. The Bregman distance provides a natural notion of error adapted to the convex penalty and is therefore the standard quantity in which convergence rates are established. Many such convergences can be seen in \cite{Hofmann2007,Bubba2025O,schuster2012regularization}.\\

    Further, we show that the theorems in \cite{JOSEPH2026102050} are special cases of the results established in this paper. Note that the analysis and convergence rates obtained therein rely on the assumption that the unit ball is compactly embedded into $C(X)$ with covering numbers satisfying a logarithmic growth condition. This makes the application rather restrictive. We are now able to remove this assumption and consequently obtain a faster convergence rate than before. Finally, the Reproducing Kernel Banach Space (RKBS) assumption used in the previous work was adopted from \cite{5179093}. This definition is based on the reflexivity of the domain space with additional assumptions on its dual space. Moreover, the Kernel function must satisfy certain continuity property to establish the compactness of the unit ball, as stated earlier. However, in this, we use a more generalized version of RKBS from \cite{Xu2023SparseBanach}, which does not require those properties. This allowed us to consider more general non-reflexive spaces with fewer assumptions without requiring a Kernel function. Also, I would like to emphasize that \cite{JMLR:v22:20-751} and \cite{MR4749129} have used the more recent definition of RKBS to establish the representer theorem, and solutions to various inverse problems independent of the reflexivity of the space. The key difference between our paper and theirs is that they have considered a functional analytic approach, whereas we use statistical approach to determine a high probability convergence.\\
    
   The notion of convergence is expressed in terms of the number of data points $m$, where the error measure should ideally converge to zero as $m\to\infty$. For that a Tikhonov-type regularization \cite{Tikhonov:1963} method is employed. Since the learning problem is based on an $L^p$-type loss $(p>1)$, the RKBS framework provides a suitable setting where point evaluations are continuous, thereby guaranteeing well-definedness and ensuring that the optimization is meaningful.\\
	
	\noindent The paper is organized as follows: Section \ref{sec2}
	introduces the problem setting, basic definitions, and necessary
	assumptions. Section \ref{sec3}
	presents the main theoretical results and corollaries. In Section \ref{proofs}, the proofs establishing
	convergence and convergence rates are given. Finally, Section \ref{sec5}
	provides a numerical example to demonstrate and validate the
	theoretical findings.

\section{Problem setting}
\label{sec2}
\noindent In this paper, we consider an injective bounded linear operator
$A$ from a real Banach space $(B_1,\parallel\cdot\parallel_{B_1})$ to
a vector space $V$ of real-valued functions over some metric space $X$, i.e
\[A:B_1\to V\subseteq \mathcal{F}(X,\mathbb{R}).\]
Then our problem will be modeled as
\[\qquad Au=g,\qquad u\in B_1~\text{and}~ g\in V,\] where
$g$ is observed only through independently and identically
distributed noisy data points $z=\{z_i\}_{i=1}^m=\{(x_i,y_i)\}_{i=1}^m\in Z^m=(X,[-M,M])^m, M>0$
that are assumed to follow an unknown Borel probability distribution
$\rho(x,y).$
The distribution will be considered to be of the form $\rho(x,y)=\rho(y|x)\rho_X(x)$,
where $\rho(y|x)$ is the conditional distribution of $y$ for some given $x$ and $\rho_X(x)$ is the marginal distribution.
\vspace{0.2cm}

\noindent For some $p>1$, we assume that there exist a minimizer $u_{\rho}\in B_1$ corresponding to the measure $\rho(x,y)$ on $Z=(X\times [-M,M])$ such that
\begin{align}\label{true}
	u_{\rho}:&=\arg \min _{u\in B_1} \int_{Z}|y-(Au)(x)|^{p}~ d \rho.
\end{align}
Further we denote
\begin{equation}\label{err}
	\mathcal{E_{\rho}}(u):= \int_{Z}|y-(Au)(x)|^{p}~ d \rho.
\end{equation}

\vspace{0.3cm}

\noindent We define the regularized solution $u_z^{\lambda}$ corresponding to the data points $z$ as

\begin{equation}\label{reg}
	u_z^{\lambda}=\text{arg}\underset{u\in B_1}{\text{min}}\left(\frac{1}{m}\sum_{i=1}^m|y_i-(Au)(x_i)|^p+\lambda\Omega(u)\right)
\end{equation}
where $\Omega$ is a continuous convex function and $\lambda>0$ is the regularization parameter. Note that we only used the vector space structure of the co-domain.

\noindent We go on to further define the $\rho$ induced regularized solution  corresponding to $\lambda$ and \eqref{true} by,
\begin{equation}\label{regtrue}
	u^{\lambda}_{\rho}:= \arg\min_{u\in B_1}\left(\mathcal{E}_{\rho}(u)+\lambda\Omega(u)\right)
\end{equation}
The existence of the above minimizers follows from the assumptions \ref{assume} (cf. \cite{BenesovaKruzik2017}, \cite{schuster2012regularization}). 

    \subsection{Definitions and basic results}
In this subsection, we give an overview of some of the basic  definitions that we used for analysis purposes from convex analysis and related literature.

\begin{dfn} Subdifferential\cite{rockafellar1970convex}:
	Let $G:B\rightarrow\mathbb{R}$ be a convex function defined on a
	Banach space $B$. The subdifferential $\partial G$ of the function $G$ at the
	point $u\in B$ is characterized as the set
    \begin{align}\label{subd}
	(\partial G)(u)&=\{ u^* \in {B}^{*}: G\left(v\right)-G(u)
	\left.\geq\left\langle v-u,u^* \right\rangle_{B\times B^*}, \forall v \in {B}\right\}\\
	& \neq \phi .\notag
\end{align}
\end{dfn}
Each element of the above set is called a subgradient. Note that here $(\partial G)(u)$ represents a set of linear functionals, and each element $\xi(u)\in(\partial G)(u)$ is a linear functional corresponding to the point $u$. That is for all $u$, $\xi(u)\in B^*$ and not to be read as $\xi$ acting on the point $u$.

\begin{dfn} Bregman distance\cite{Bregman1967}:
	Let $B$ be a Banach space, and $G$ be a real convex function on $B$. If $u$ and $v$ are arbitrary elements of $B$, then the Bregman distance from $v$ to $u$ corresponding to $G$ with respect to some $\xi(u)\in(\partial G)(u)$ is defined as 
	\begin{equation}\label{breg}
		D_G^{\xi(u)}(v,u)=G(v)-G(u)-\langle v-u,\xi(u)\rangle_{B\times B^*}
	\end{equation}
\end{dfn}
We can easily infer from \eqref{subd} and \eqref{breg} that the Bregman distance is always non negative. Further, for any subgradients $\eta(v)\in(\partial G)(v)$ and $\pi(u)\in(\partial G)(u)$, it follows that

\begin{equation*}
 D_{G}^{\eta(v)}(u,v)+D_{G}^{\pi(u)}(v,u)=\langle u-v,\pi(u)-\eta(v)\rangle_{B\times B^*}.
\end{equation*}
The non-negativity of the Bregman distance then ensures that 
\begin{equation}\label{ineq}
D_{G}^{\eta(v)}(u,v)\leq \|u-v\|_{B}\|\eta(v)-\pi(u)\|_{B^*}.
\end{equation}

\begin{dfn}\label{RKBS}{Reproducing Kernel Banach Space}\cite{Xu2023SparseBanach}:
A Banach space $B$ of functions defined on a prescribed set $X$ is called an RKBS if point-evaluation functionals
$\delta_x$, for all $x\in X$, are continuous on $B$, that is, for each $x\in X$ there exists a constant $c_x > 0$ such that
\begin{equation*}
	|\delta_x(f)|\leq c_x\|f\|_{B}~~\text{for all}~~f\in B
\end{equation*}
\end{dfn}

\begin{dfn} $q$-Uniform Convex Space:
	Let $J_q(u):=(\partial G)(u)$ denote the subdifferential at the point $u\in B$ when $G(u)=\frac{1}{q}\|u\|_{B}^q$. Then the Banach space ${B}$ is called a $q$-uniform convex space if there exists a constant $c_{q}>0$ such that
	\begin{equation}\label{qcon}
		\begin{array}{r}
			\frac{1}{q}\left(\|v\|_{{B}}^{q}-\|u\|_{{B}}^{q}\right) \geq\left\langle v-u,j_q(u)\right\rangle_{B^*\times B}+c_{q}\|v-u\|_{{B}}^{q}
		\end{array}
	\end{equation}
	for all $u$, $v \in {B}$ and all $j_{q}(u) \in J_{q}(u)$ \cite{Xu1991Inequalities}.
	Equivalently, it means that
	\begin{equation}\label{qcon2}
		c_{q}\|v-u\|_{{B}}^{q}\leq D_G^{j_q(u)}(v,u)
	\end{equation}
\end{dfn}
\noindent  Hilbert spaces are
$2$-uniform convex, which follows directly from the parallelogram law. From \cite{Hanner1956,XuRoach1991}, we can ensure that Banach spaces such as $l^p,~L^P$ are $p$-uniform convex when $p\geq2$ and $2$-uniform for $1<p<2$.

  \subsection{Assumptions}\label{assume1}
We now make the following assumptions:
\begin{enumerate}
	\item[(A1)] The evaluation functionals $S_x:B_1\to\mathbb{R}$ with respect to the sample points $x\in X$
	\begin{equation*}
		S_x(u)=(Au)(x)
	\end{equation*}
	are uniformly continuous, that is, there exists a constant $0<k<\infty$ such that
	\begin{equation*}
		{|S_x(u)|\leq k\n u_{B_1}}\quad\forall x\in X
	\end{equation*}
Moreover, the map $u\mapsto (Au)(x)$ is measurable.
\end{enumerate}
\textbf{Remark:} The image space of $A$, equipped with the norm induced from the pre-image, makes $A(B_1)$ an RKBS. 
\begin{enumerate}
	\item[(A2)] We assume that the approximation error satisfies
	\begin{equation}\label{eq9}
		\|u^{\lambda}_{\rho}-u_{\rho}\|_{B_1}\leq c_{\beta}\lambda^{\beta}
	\end{equation}
	where $\beta$ and $c_{\beta}$ are positive constants.
\end{enumerate}
Such bounds arising from the second assumption are frequently encountered in the study of inverse problems (cf. Chapter $4.4$ of \cite{nair2009linear}, Proposition $5.8$ in
\cite{BGM}, \cite{Hofmann2007}, and \cite{JOSEPH2026102050}).\\

\noindent Further, we make the following assumptions on the nature of $\Omega$.
\begin{enumerate}\label{assume}
	\item[(B1)] $\Omega$ is a proper convex function on $B_1$.
	\item [(B2)] There exists a topology $W$ on $B_1$ by which both $S_x$ and $\Omega$ are lower semicontinuous.
	\item[(B3)] The sublevel set of $\Omega$ for any constant $c>0$, i.e.,
	\begin{equation*}
		\mathcal{M}_c(\Omega):=\{u\in B_1~|~\Omega(u)\leq c\}
	\end{equation*}
	 is pre compact in the topology $W$.
\end{enumerate}
In the numerical illustration part, we consider the topology $W$ to be the weak* topology. Also, the minimizer is unique if the minimizing functional in \eqref{reg} and \eqref{regtrue} are strictly convex.

\section{Main results}\label{sec3}
In this section, we present the major results associated with
the convergence and convergence rate of the proposed scheme. We will state the theorems and the associated corollary.

\begin{theorem}\label{thm1}
	Let $(B_1,\|\cdot\|_{B_1})$ be a Banach space and $V$ be a real vector space of functions
	that maps a metric space $X$ to $\mathbb{R}$. Consider $A:B_1\to V$ to be an injective linear operator
	such that the assumptions \eqref{assume1} are true. Then there exists a $\xi(u_{z}^{\lambda})\in (\partial\Omega)(u_{z}^{\lambda})$ such that
	\begin{equation*}
		\xi\left(u_{z}^{\lambda}\right)=\frac{p}{m\lambda}\sum_{i=1}^{m}S_{x_i}^*\Big( \left|y_i-(Au_{z}^{\lambda})(x_i)\right|^{p-1}\text{sgn}[y_i-(Au_{z}^{\lambda})(x_i)]\Big),
	\end{equation*}
	and $\forall \epsilon>0$, we have
	\begin{align*}
		\underset{z\in Z^m}{Prob}\left\{D_{\Omega}^{\xi(u_{z}^{\lambda})}(u_{\rho},u_{z}^{\lambda})\leq(M_{\lambda}+R_{\lambda})c_{\beta}\lambda^{\beta}+\epsilon\right\}\geq 1-2\exp\left\{-\frac{mp^2\epsilon^2}{2\lambda^2\omega_{\lambda}^2}\right\}
	\end{align*}
	where $R_{\lambda},~M_{\lambda}$, and $\omega_{\lambda}$ are positive constants that depend only on $\lambda$.
\end{theorem}
	\vspace{0.3cm}
The next theorem is for the case when $\Omega(u)=\|u\|_{B_1}$. It is evident that such penalty functions have been used in the literature; see \cite{Lorenz2008}, \cite{doi:10.1137/24M1719785}.
\begin{theorem}\label{thm2}
	Let $\delta\in(0,1)$ and suppose that the assumptions of Theorem \ref{thm1} hold true. Then, with a confidence of $1-\delta$, we have
	
	\begin{align}\label{ineq22}
		D_{\Omega}^{\xi(u_{z}^{\lambda})}(u_{\rho},u_{z}^{\lambda})\leq(M_{\lambda}+R_{\lambda})c_{\beta}\lambda^{\beta}+\frac{\sqrt{2\ln(2/\delta)}p\omega_{\lambda}}{\lambda\sqrt{m}}
	\end{align}

	Further, assume that $\Omega(u)=\|u\|_{B_1}$. Then with a confidence of $1-\delta$, we have the following convergence 
	\begin{equation}
		D_{\Omega}^{\xi(u_{z}^{\lambda})}(u_{\rho},u_{z}^{\lambda})\leq4c_{\beta}\lambda^{\beta}+{2p{\sqrt{2\ln(2/\delta)}}}~\frac{\left[\lambda\sigma_{\max}+kM^p\right]^{p}}{\lambda^{p+1}\sqrt{m}}
	\end{equation}
	where $\sigma_{\max}=\max\{k\|u_{\rho}\|_{B_1},M\}$.
\end{theorem}
\vspace{0.3cm}
\noindent The following corollary is now immediate.
\begin{Corollary}\label{cor3}
	Let $\delta\in(0,1)$ and suppose that the assumptions of Theorem \ref{thm2} hold true. If $m$ is large enough, then for
	\begin{equation}\label{alp}
		\lambda={m^{-\frac{1}{2(p+1+\beta)}}}
	\end{equation}
	with a confidence of $1-\delta$, we have the following convergence rate
	\begin{equation}\label{conv}
		D_{\Omega}^{\xi(u_{z}^{\lambda})}(u_{\rho},u_{z}^{\lambda})=O(\lambda^{\beta}).
	\end{equation}	
\end{Corollary}
\vspace{0.3cm}
So far, we have considered a generic Banach space. We will show that when $B_1$ is a $q$- uniform convex space, we are able to obtain a sharper convergence rate than before. We use $\Omega(u)=\frac{1}{q}\|u\|_{B_1}^q,q>1$, to obtain the convergence in terms of norm. Since uniformly convex spaces are reflexive, the minimizers $u_{\rho}^{\lambda}$ and $u_z^{\lambda}$ are well defined since the right-hand side minimizes a strictly convex, continuous, and coercive function (\cite{MR463994}, proposition
$1.2$).

\begin{theorem}\label{thm3}
	Let $(B_1,\|\cdot\|_{B_1})$ be a $q$-uniform convex Banach space and $V$ be a real vector space of functions
	that maps a metric space $X$ to $\mathbb{R}$. Consider $A:B_1\to V$ to be an injective linear operator
	such that the assumptions \eqref{assume1} are true. Define $\Omega(u)$ to be $\frac{1}{q}\|u\|_{B_1}^q$. Then 
	\begin{align}\label{ineq27}
		\underset{z\in Z^m}{\text{Prob}}\left\{\left\|u_z^{\lambda}-u_{\rho}\right\|_{B_1}\leq c_{\beta}{\lambda}^{\beta}+\epsilon\right\} 
		\geq 1-2 \exp \left(-\frac{m \lambda^{2} c_{q}^{2} \epsilon^{2(q-1)}}{2p^2
			\omega^{2}}\right)
	\end{align}
	where $\omega$ is a positive constant independent of $\lambda$.
\end{theorem}
\vspace{0.3cm}
Finally, we have a corollary of the above result.
\begin{Corollary}\label{cor2}
	Let $\delta\in(0,1)$ and suppose that the assumptions of Theorem \ref{thm3} hold true. Then, for
	\begin{equation}\label{lambda}
		\lambda={m^{-\frac{1}{2(1+\beta(q-1))}}}
	\end{equation}
	with a confidence of $1-\delta$, we have the following convergence rate
	\begin{equation}\label{conv}
		\left\|u_z^{\lambda}-u_{\rho}\right\|_{B_1}=O(\lambda^{\beta}).
	\end{equation}
\end{Corollary}

 \section{Proofs}
\label{proofs}
Now we present the proofs for the above stated theorems and corollary. We obtain a high probability upper bound for the convergence of the Bregman distance in terms of the number of data points $m$. The proofs are presented as a sequence of lemmas. In lemmas \ref{lemma1} and \ref{lemma2}, we establish an explicit form for $\partial\mathcal{E}_{\rho}$ and $(\partial \Omega)(u_{\rho}^{\lambda})$ as before. Lemma \ref{lemma3} shows an upper bound for the Bregman difference between $u_{\rho}^{\lambda}$ and $u_{\gamma}^{\lambda}$, where the latter is the regularized solution w.r.t the distribution $\gamma$. By the end of the proof of theorem \ref{thm1}, we consider $\gamma$ to be the empirical distribution to obtain a bound as given by the theorems.

Now we start with the following lemma that gives an explicit form for the subdifferential of $\mathcal{E_{\rho}}$ at any point $u\in B_1$.

\begin{lemma}\label{lemma1}
	Let $(B_1,\|\cdot\|_{B_1})$ be a Banach space and $V$ be a real vector space of functions
	that maps a compact metric space $X$ to $\mathbb{R}$. Let $A:B_1\to V$ be an injective linear operator
	such that the assumptions \ref{assume1} are true.
	If $\mathcal{E}_{\rho}$ is defined as in \eqref{err}, then 
	for any $u_{0} \in {B_1} $
	\begin{equation}
		(\partial \mathcal{E}_{\rho})\left(u_{0}\right)=\left\{ -p\int_{Z}S_x^*\Big(\left|y-(Au_0)(x)\right|^{p-1}\text{sgn}[y-(Au_0)(x)]\Big) d
		\rho\right\}.
	\end{equation}
\end{lemma}
 \begin{proof}
	We begin with a function $F:B_1\to\mathbb{R}$ identified as
	\begin{equation*}
		F(u)=\left|y-(Au)(x)\right|^p=\left|y-S_x(u)\right|^p
	\end{equation*}
	
	\noindent We use Proposition $7.8$ from \cite{Showalter1997} to show that the subdifferential set of $F$ corresponding to the point $u_0$ is
\begin{equation*}
	(\partial F)(u_0)=\left\{-S_x^*~(\partial
	I){[y-(Au_0)(x)]}~\right\}.
\end{equation*}

\noindent   Here $I(r)=|r|^p,~I:\mathbb{R}\to\mathbb{R}$ is a real-valued differentiable function acting at the point $y-(Au_0)(x)$.
Now we employ Theorem $25.1$ from \cite{rockafellar1970convex}. Since the gradient of $I$ at the point $r$ is of the form $p|r|^{p-1}\text{sgn}[r]$, we get
	
	\begin{equation}\label{subF}
		(\partial F)(u_0)=\left\{-S_x^*\Big(p
		\left|y-(Au_0)(x)\right|^{p-1}\text{sgn}[y-(Au_0)(x)]\Big)\right\}
	\end{equation}
	 Note that $(\partial F)(u_0)$ is a singleton set, therefore without loss of generality we use it to denote the unique subgradient itself. \\
	
	\noindent   Now, by the definition of subgradient,
	\begin{equation*}
		F(v)-F(u_0)\geq\langle v-u_0,(\partial F)(u_0)\rangle_{B_1\times B_1^*}\quad\forall v\in
		B_1.
	\end{equation*}
	
	\noindent   Integrate both sides of the inequality. By the monotonicity of the integral, we get
	\begin{align*}
		\int_Z F(v)~d\rho-\int_Z F(u_0)~d\rho&\geq\int_Z\langle v-u_0,(\partial F)(u_0)\rangle_{B_1\times B_1^*} d\rho\\
		&=\int_Z(\partial F)(u_0)(v-u_0)d\rho .
	\end{align*}
	
	\noindent   We replace the integral on the left by $\mathcal{E_{\rho}}$. Further, the RHS of the above inequality is a continuous linear functional from $B_1$ to $\mathbb{R}$ which can be written as

	\begin{align*}
	\mathcal{E_{\rho}}(v)-\mathcal{E_{\rho}}(u_0) \geq\Big\langle v-u_0,\int_Z(\partial F)(u_0)~d\rho\Big\rangle_{B_1\times
			B_1^*}.
	\end{align*}
	
	\noindent Therefore, using the definition of the subdifferential together with \eqref{subF}, we conclude that
	\begin{equation*}
		(\partial \mathcal{E}_{\rho})\left(u_{0}\right)=\left\{-p \int_{Z}S_x^*\Big( \left|y-(Au_0)(x)\right|^{p-1}\text{sgn}[y-(Au_0)(x)]\Big) d
		\rho\right\}.
	\end{equation*}
	This proves the lemma.
\end{proof}
\vspace{0.1cm}

 \begin{lemma}\label{lemma2}
	Suppose that $\Omega$ is as in \eqref{reg} and $u_{\rho}^{\lambda}$ be defined as in \eqref{regtrue}.
	Then, there exists a $\xi(u_{\rho}^{\lambda}) \in (\partial\Omega)\left(u_{\rho}^{\lambda}\right)$ which satisfies the equation
	\begin{equation}
		\lambda \xi(u_{\rho}^{\lambda})=p \int_{Z}S_x^*\Big( \left|y-(Au_{\rho}^{\lambda})(x)\right|^{p-1}\text{sgn}[y-(Au_{\rho}^{\lambda})(x)]\Big) d
		\rho .
	\end{equation}
\end{lemma}
\begin{proof}
	By definition, ${u_{\rho}^{\lambda}}$ is the minimizer of the function $\mathcal{E_{\rho}}(u)+\lambda\Omega(u)$. Therefore, it follows that
	\begin{equation}\label{eq12}
		0 \in
		\partial\left(\mathcal{E}_{\rho}+{\lambda}\Omega\right)(u_{\rho}^{\lambda})
	\end{equation}
	Theorem $47B$ of (\cite{zeidler1985nonlinear}) gives us
	\begin{align*}
     \partial\left(\mathcal{E}_{\rho}+{\lambda}\Omega\right)(u_{\rho}^{\lambda})=\left(\partial\mathcal{E}_{\rho}\right)(u_{\rho}^{\lambda})+\lambda\left(\partial\Omega\right)(u_{\rho}^{\lambda})
		.
	\end{align*}

	\noindent Hence, by Lemma \ref{lemma1} and \eqref{eq12},
	\begin{equation*}
		0 \in\left\{-p \int_{Z}S_x^*\Big( \left|y-(Au_{\rho}^{\lambda})(x)\right|^{p-1}\text{sgn}[y-(Au_{\rho}^{\lambda})(x)]\Big) d \rho\right\}+\lambda(\partial\Omega)\left(u_{\rho}^{\lambda}\right) .
	\end{equation*}
	
	\noindent Therefore, there exist some $\xi(u_{\rho}^{\lambda}) \in (\partial\Omega)\left(u_{\rho}^{\lambda}\right)$ such that
	\begin{equation*}
		\xi\left(u_{\rho}^{\lambda}\right) =\frac{p}{\lambda
		} \int_{Z}S_x^*\Big( \left|y-(Au_{\rho}^{\lambda})(x)\right|^{p-1}\text{sgn}[y-(Au_{\rho}^{\lambda})(x)]\Big) d \rho.
	\end{equation*}
	This proves the lemma.
\end{proof}
\vspace{0.1cm}
\begin{lemma}\label{lemma3}
	Let $B_1$ be a Banach space and $\mathcal{E}_{\rho}(u)$ be as defined \eqref{err}.
	If $u_{\rho}^{\lambda}$ and $u_{\gamma}^{\lambda}$ are the minimizers given by \eqref{regtrue} for the distributions $\rho$ and $\gamma$, respectively, then,
\begin{align*}
	\frac{\lambda}{p} D_{\Omega}^{\xi(u_{\gamma}^{\lambda})}(u_{\rho}^{\lambda},u_{\gamma}^{\lambda}) \leq
	\Big\langle u_{\rho}^{\lambda}-u_{\gamma}^{\lambda}, \int_{Z}S_x^*(W_{\gamma}^{\lambda}(x,y) ) d \rho-\int_{Z}S_x^*(W_{\gamma}^{\lambda}(x,y) ) d \gamma \Big\rangle_{B_1\times
		B_1^*}
\end{align*}
	where 
	\begin{equation*}
	W_{\gamma}^{\lambda}(x,y)=\left|y-(Au_{\gamma}^{\lambda})(x)\right|^{p-1}\text{sgn}[y-(Au_{\gamma}^{\lambda})(x)]
	\end{equation*}
 and $\xi(u_{\gamma}^{\lambda})$ is the subgradient defined by lemma \ref{lemma2} for the distribution $\gamma$.
	\end{lemma}
	\begin{proof}
By using Lemma \ref{lemma1} together with  \eqref{subd}, we can state that
\begin{align}\label{eq39}
	&\mathcal{E}_{\rho}\left(u_{\rho}^{\lambda}\right)-\mathcal{E}_{\rho}\left(u_{\gamma}^{\lambda}\right)
	\geq\notag\\
	&\Big\langle u_{\rho}^{\lambda}-u_{\gamma}^{\lambda},-p \int_{Z}S_x^*\Big( \left|y-(Au_{\gamma}^{\lambda})(x)\right|^{p-1}\text{sgn}[y-(Au_{\gamma}^{\lambda})(x)]\Big) d
	\rho\Big\rangle_{B_1\times
		B_1^*}.
\end{align}
Further, by the definitions of $u_{\rho}^{\lambda}$ and $u_{\gamma}^{\lambda}$, we can claim that
\begin{align*}
	&~\mathcal{E}_{\rho}\left(u_{\rho}^{\lambda}\right)+{\lambda}\Omega(u_{\rho}^{\lambda})\leq\mathcal{E}_{\rho}\left(u_{\gamma}^{\lambda}\right)+{\lambda}\Omega(u_{\gamma}^{\lambda})\\
	\Rightarrow & ~\mathcal{E}_{\rho}\left(u_{\rho}^{\lambda}\right)-\mathcal{E}_{\rho}\left(u_{\gamma}^{\lambda}\right)+{\lambda}\left(\Omega(u_{\rho}^{\lambda})-\Omega(u_{\gamma}^{\lambda})\right)\leq
	0.
\end{align*}
From \eqref{eq39}, the above inequality becomes
\begin{align*}
	 \Big\langle u_{\rho}^{\lambda}-u_{\gamma}^{\lambda},-p \int_{Z}S_x^*\Big( \left|y-(Au_{\gamma}^{\lambda})(x)\right|^{p-1}\text{sgn}[y-(&Au_{\gamma}^{\lambda})(x)]\Big) d
	\rho\Big\rangle_{B_1\times
		B_1^*} \\
	&
	+{\lambda}\left(\Omega(u_{\rho}^{\lambda})-\Omega(u_{\gamma}^{\lambda})\right)\leq 0
\end{align*}
Replace the term $\left|y-(Au_{\gamma}^{\lambda})(x)\right|^{p-1}\text{sgn}[y-(Au_{\gamma}^{\lambda})(x)]$ by $W_{\gamma}^{\lambda}(x,y)$    
  \begin{equation*}
	 p\Big\langle u_{\rho}^{\lambda}-u_{\gamma}^{\lambda},-\int_{Z}S_x^*(W_{\gamma}^{\lambda}(x,y) ) d \rho\Big\rangle_{B_1\times B_1^*}
	+{\lambda}\left(\Omega(u_{\rho}^{\lambda})-\Omega(u_{\gamma}^{\lambda})\right)\leq 0
\end{equation*}

\noindent Choose $\xi(u_{\gamma}^{\lambda})$ to be the subgradient in $(\partial\Omega)(u_{\gamma}^{\lambda})$ that satisfies lemma \ref{lemma2} with respect to the distribution $\gamma$. Then by \eqref{breg}, the inequality takes the form

\begin{align*}
	 & ~p\left\langle u_{\rho}^{\lambda}-u_{\gamma}^{\lambda},-\int_{Z}S_x^*(W_{\gamma}^{\lambda}(x,y) ) d \rho\right\rangle_{B_1\times B_1^*} \\
	&~~~~~~~~~ +p\left\langle u_{\rho}^{\lambda}-u_{\gamma}^{\lambda}, \int_{Z}S_x^*(W_{\gamma}^{\lambda}(x,y) ) d \gamma\right\rangle_{B_1\times B_1^*}
	+\lambda D_{\Omega}^{\xi(u_{\gamma}^{\lambda})}(u_{\rho}^{\lambda},u_{\gamma}^{\lambda})\leq 0.
\end{align*}
Finally, we rearrange the terms to get the following form
\begin{align*}
\frac{\lambda}{p} D_{\Omega}^{\xi(u_{\gamma}^{\lambda})}(u_{\rho}^{\lambda},u_{\gamma}^{\lambda}) \leq
	\Big\langle u_{\rho}^{\lambda}-u_{\gamma}^{\lambda}, \int_{Z}S_x^*(W_{\gamma}^{\lambda}(x,y) ) d \rho-\int_{Z}S_x^*(W_{\gamma}^{\lambda}(x,y) ) d \gamma \Big\rangle_{B_1\times
		B_1^*}
\end{align*}
This completes the proof.
\end{proof}
\vspace{0.3cm}

\begin{lemma}\label{lemma4}(cf. \cite{CuckerZhou2007}, Corollary $3.6$). Let $h$ be a random variable on a probability space $Z$ with mean
	$\mu$ and finite variance, and satisfying $|h(z)-\mu|\leq N$ for
	almost all $z\in Z$. Then for all $\epsilon>0$,
	\begin{equation*}
		\underset{z\in Z^m}{\text{Prob}}    ~\Big\{ \frac{1}{m}\sum_{i=1}^{m}h(z_i)-\mu\geq\epsilon\Big\}\leq \exp\Big\{-\frac{m\epsilon^2}{2N^2}\Big\}.
	\end{equation*}
\end{lemma}
\vspace{0.1cm}

\noindent From the above lemma, it can be easily shown that 
\begin{equation}\label{cor}
	\underset{z\in Z^m}{\text{Prob}}    ~\Big\{\Big| \frac{1}{m}\sum_{i=1}^{m}h(z_i)-\mu\Big|\leq\epsilon\Big\}\geq 1-2\exp\Big\{-\frac{m\epsilon^2}{2N^2}\Big\}.
\end{equation}
\vspace{0.4cm}

	\noindent We are now ready to present the proof of theorem \ref{thm1}.

\begin{proof}[\textbf{Proof of theorem \ref{thm1}}:]
	The Bregman distance from $u_{\rho}$ to $u_{\gamma}^{\lambda}$ with respect to $\xi(u_{\gamma}^{\lambda})\in (\partial\Omega)(u_{\gamma}^{\lambda})$ is
	\begin{equation}
		D_{\Omega}^{\xi(u_{\gamma}^{\lambda})}(u_{\rho},u_{\gamma}^{\lambda})=\Omega(u_{\rho})-\Omega(u_{\gamma}^{\lambda})-\langle u_{\rho}-u_{\gamma}^{\lambda},\xi(u_{\gamma}^{\lambda})\rangle_{B_1\times B_1^*}.
	\end{equation}
	Add and subtract $\Omega(u_{\rho}^{\lambda})$ to the RHS
	\begin{align*}
		D_{\Omega}^{\xi(u_{\gamma}^{\lambda})}(u_{\rho},u_{\gamma}^{\lambda})=\Omega(u_{\rho})-\Omega(u_{\rho}^{\lambda})+\Omega(u_{\rho}^{\lambda})-\Omega(u_{\gamma}^{\lambda})-\langle u_{\rho}-u_{\gamma}^{\lambda},\xi(u_{\gamma}^{\lambda})\rangle_{B_1\times B_1^*}.
	\end{align*}
Let $\zeta(u_{\rho}^{\lambda})$ be some element $(\partial \Omega)(u_{\rho}^{\lambda})$. Substitute $\Omega(u_{\rho}^{\lambda})-\Omega(u_{\gamma}^{\lambda})$ and $\Omega(u_{\rho})-\Omega(u_{\rho}^{\lambda})$ by the Bregman distances with respect to $\xi(u_{\gamma}^{\lambda})$ and $\zeta(u_{\rho}^{\lambda})$, respectively, to obtain

 \begin{align*}
		D_{\Omega}^{\xi(u_{\gamma}^{\lambda})}(u_{\rho},u_{\gamma}^{\lambda})=	D_{\Omega}^{\zeta(u_{\rho}^{\lambda})}(u_{\rho},u_{\rho}^{\lambda})+&\langle u_{\rho}-u_{\rho}^{\lambda},\zeta(u_{\rho}^{\lambda})\rangle_{B_1\times B_1^*}+D_{\Omega}^{\xi(u_{\gamma}^{\lambda})}(u_{\rho}^{\lambda},u_{\gamma}^{\lambda})\\
		&
	    \langle  u_{\rho}^{\lambda}-u_{\gamma}^{\lambda},\xi(u_{\gamma}^{\lambda})\rangle_{B_1\times B_1^*}-\langle u_{\rho}-u_{\gamma}^{\lambda},\xi(u_{\gamma}^{\lambda})\rangle_{B_1\times B_1^*}.
	\end{align*}
	which, after simplification will yield
	\begin{align*}
		D_{\Omega}^{\xi(u_{\gamma}^{\lambda})}(u_{\rho},u_{\gamma}^{\lambda})=	D_{\Omega}^{\zeta(u_{\rho}^{\lambda})}(u_{\rho},u_{\rho}^{\lambda})+D_{\Omega}^{\xi(u_{\gamma}^{\lambda})}(u_{\rho}^{\lambda},u_{\gamma}^{\lambda})+\langle  u_{\rho}-u_{\rho}^{\lambda},\zeta(u_{\rho}^{\lambda})-\xi(u_{\gamma}^{\lambda})\rangle_{B_1\times B_1^*}
	\end{align*}
From \eqref{ineq}, we deduce that for any $\tau(u_{\rho})\in\partial\Omega(u_{\rho})$
	
		\begin{align*}
		D_{\Omega}^{\xi(u_{\gamma}^{\lambda})}(u_{\rho},u_{\gamma}^{\lambda})\leq	\|u_{\rho}-u_{\rho}^{\lambda}\|_{B_1}\|\tau(u_{\rho})-&\zeta(u_{\rho}^{\lambda})\|_{B_1^*}+D_{\Omega}^{\xi(u_{\gamma}^{\lambda})}(u_{\rho}^{\lambda},u_{\gamma}^{\lambda})\\
		&+\langle  u_{\rho}-u_{\rho}^{\lambda},\zeta(u_{\rho}^{\lambda})-\xi(u_{\gamma}^{\lambda})\rangle_{B_1\times B_1^*}.
	\end{align*}
	Fix such a $\tau(u_{\rho})$. Let $R_{\lambda}^{\rho}$ and $M_{\lambda}^{\rho,\gamma}$ be such that $\|\tau(u_{\rho})-\zeta(u_{\rho}^{\lambda})\|_{B_1^*}\leq R_{\lambda}^{\rho}$ and $\|\zeta(u_{\rho}^{\lambda})-\xi(u_{\gamma}^{\lambda})\|_{B_1^*}\leq M_{\lambda}^{\rho,\gamma}$. Then the LHS can be bounded in the following way,
		\begin{align*}
		D_{\Omega}^{\xi(u_{\gamma}^{\lambda})}(u_{\rho},u_{\gamma}^{\lambda})\leq	\|u_{\rho}-u_{\rho}^{\lambda}\|_{B_1}R_{\lambda}^{\rho}+D_{\Omega}^{\xi(u_{\gamma}^{\lambda})}(u_{\rho}^{\lambda},u_{\gamma}^{\lambda})
		&+\| u_{\rho}-u_{\rho}^{\lambda}\|_{B_1}M_{\lambda}^{\rho,\gamma}
	\end{align*}
	Let's consider $\gamma$ to be the empirical measure $\gamma_z$. Then $u_{\gamma}^{\lambda}$ coincides with $u_{z}^{\lambda}$. Further, we will drop the notation ${\rho,\gamma_z}$ for $M_{\lambda}^{\rho,\gamma_z}$ and $R_{\lambda}^{\rho}$ since both the measure $\rho$ and $\gamma_z$ will stay invariant throughout the paper. We now have the inequality
	\begin{align*}
	D_{\Omega}^{\xi(u_{z}^{\lambda})}(u_{\rho},u_{z}^{\lambda})-\|u_{\rho}-u_{\rho}^{\lambda}\|_{B_1}(M_{\lambda}+R_{\lambda})\leq	D_{\Omega}^{\xi(u_{z}^{\lambda})}(u_{\rho}^{\lambda},u_{z}^{\lambda})
\end{align*}
By assumption $2$, we get $-\|u_{\rho}-u_{\rho}^{\lambda}\|_{B_1}\geq-c_{\beta}\lambda^{\beta}$, and hence
	\begin{align}\label{ineq19}
	D_{\Omega}^{\xi(u_{z}^{\lambda})}(u_{\rho},u_{z}^{\lambda})-(M_{\lambda}+R_{\lambda})c_{\beta}\lambda^{\beta}\leq	D_{\Omega}^{\xi(u_{z}^{\lambda})}(u_{\rho}^{\lambda},u_{z}^{\lambda})
\end{align}
To prove the theorem, we need to find an upper bound for the RHS. For that we use lemma \ref{lemma3}. Replace the measure $\gamma$ with the empirical measure $\gamma_z$ in lemma \ref{lemma3} to obtain

\begin{align*}
	 D_{\Omega}^{\xi(u_{z}^{\lambda})}(u_{\rho}^{\lambda},u_{z}^{\lambda}) \leq
	\frac{p}{\lambda}\Big\langle u_{\rho}^{\lambda}-u_{z}^{\lambda}, \int_{Z}S_x^*(W_{\gamma_z}^{\lambda}(x,y) ) d \rho-\frac{1}{m}\sum_{i=1}^{m}S_{x_i}^*(W_{\gamma_z}^{\lambda}(x_i,y_i) )\Big\rangle_{B_1\times
		B_1^*}
\end{align*}
Simplify the terms to get
\begin{align*}
   D_{\Omega}^{\xi(u_{z}^{\lambda})}(u_{\rho}^{\lambda},u_{z}^{\lambda}) \leq
	\int_{Z}\frac{p}{\lambda}\Big\langle u_{\rho}^{\lambda}-u_{z}^{\lambda}, &S_x^*(W_{\gamma_z}^{\lambda}(x,y) ) \Big\rangle_{\small{B_1\times
			B_1^*}}d \rho\\
	&-\frac{1}{m}\sum_{i=1}^{m}\frac{p}{\lambda}\Big\langle u_{\rho}^{\lambda}-u_{z}^{\lambda},S_{x_i}^*(W_{\gamma_z}^{\lambda}(x_i,y_i) ) \Big\rangle_{\small{B_1\times
			B_1^*}}
\end{align*}
which reduces to the form
\begin{align*}
D_{\Omega}^{\xi(u_{z}^{\lambda})}(u_{\rho}^{\lambda},u_{z}^{\lambda}) \leq& \int_{Z}h(x,y) d \rho-\frac{1}{m}\sum_{i=1}^{m}h(x_i,y_i)\\
\leq&\left|\int_{Z}h(x,y) d \rho-\frac{1}{m}\sum_{i=1}^{m}h(x_i,y_i)\right| .
\end{align*}
where we used $h(x,y)$ for $\frac{p}{\lambda}\langle u_{\rho}^{\lambda}-u_{z}^{\lambda}, S_x^*(W_{\gamma_z}^{\lambda}(x,y) ) \rangle_{\small{B_1\times
		B_1^*}}$. Hence, from \eqref{ineq19} it follows that
\begin{equation}\label{ineq20}
		D_{\Omega}^{\xi(u_{z}^{\lambda})}(u_{\rho},u_{z}^{\lambda})-(M_{\lambda}+R_{\lambda})c_{\beta}\lambda^{\beta}\leq \left|\int_{Z}h(x,y) d \rho-\frac{1}{m}\sum_{i=1}^{m}h(x_i,y_i)\right| .
\end{equation}

We will apply lemma \ref{lemma4} to the expression on the right. For that we need to show $|h(x,y)-\int h(x,y)d\rho|$ is bounded. Consider
\begin{align*}
	|h(x,y)|=&\left|\frac{p}{\lambda}\langle u_{\rho}^{\lambda}-u_{z}^{\lambda}, S_x^*(W_{\gamma_z}^{\lambda}(x,y) ) \rangle_{\small{B_1\times B_1^*}} \right|\\
	\leq&\frac{p}{\lambda}|S_x(u_{\rho}^{\lambda}-u_{z}^{\lambda})|\left|y-(Au_{z}^{\lambda})(x)\right|^{p-1}|\text{sgn}[y-(Au_{z}^{\lambda})(x)]|\\
	\leq &\frac{pk}{\lambda}\|u_{\rho}^{\lambda}-u_{z}^{\lambda}\|_{B_1}\left|y-(Au_{z}^{\lambda})(x)\right|^{p-1}\\
	\leq&\frac{pk}{\lambda}(\|u_{\rho}^{\lambda}\|_{B_1}+\|u_{z}^{\lambda}\|_{B_1})[M+k\|u_z^{\lambda}\|_{B_1}]^{p-1}
	\leq\frac{p\omega_{\lambda}'}{\lambda}
\end{align*}
Where we denote $k(\|u_{\rho}^{\lambda}\|_{B_1}+\|u_{z}^{\lambda}\|_{B_1})[M+k\|u_z^{\lambda}\|_{B_1}]^{p-1}$ by $\omega_{\lambda}'$. Then the following inequality follows
\begin{equation*}
	|h(x,y)|\leq\frac{p\omega_{\lambda}'}{\lambda}.
\end{equation*}
Therefore, we have, $|h(x,y)-\int h(x,y)d\rho|\leq |h(x,y)|+\int|h(x,y)|d\rho\leq \frac{p\omega_{\lambda}}{\lambda}$, where $\omega_{\lambda}$ denotes $2\omega_{\lambda}'$. This shows that the required term is bounded for each ${\lambda}$. Then, lemma \ref{lemma4} gives us
\begin{align}\label{ineq21}
	\underset{z\in Z^m}{\text{Prob}}\left\{\left|\int_Z h(x,y)d\rho-\frac{1}{m}\sum_{i=1}^{m}h(x_i,y_i)\right|\leq\epsilon\right\}\geq 1-2\exp\left\{-\frac{m\lambda^2\epsilon^2}{2p^2\omega_{\lambda}^2}\right\}
	\end{align}
Now \eqref{ineq20} tells us that
\begin{align*}
	\{z\in Z^m:D_{\Omega}^{\xi(u_{z}^{\lambda})}(u_{\rho},u_{z}^{\lambda})-(M_{\lambda}+R_{\lambda})c_{\beta}\lambda^{\beta}\leq\epsilon\}\supseteq\\
	\left\{z\in Z^m:\left|\int_Z h(x,y)d\rho-\frac{1}{m}\sum_{i=1}^{m}h(x_i,y_i)\right|\leq\epsilon\right\}
\end{align*}
In terms of measure, this would imply that
\begin{align*}
	\underset{z\in Z^m}{Prob}~\{D_{\Omega}^{\xi(u_{z}^{\lambda})}(u_{\rho},u_{z}^{\lambda})-&(M_{\lambda}+R_{\lambda})c_{\beta}\lambda^{\beta}\leq\epsilon\}~\geq~\\
	&\underset{z\in Z^m}{\text{Prob}}\left\{\left|\int_Z h(x,y)d\rho-\frac{1}{m}\sum_{i=1}^{m}h(x_i,y_i)\right|\leq\epsilon\right\}
\end{align*}
Hence, by \eqref{ineq21}, the theorem is complete.
\begin{align*}
	\underset{z\in Z^m}{Prob}\left\{D_{\Omega}^{\xi(u_{z}^{\lambda})}(u_{\rho},u_{z}^{\lambda})\leq(M_{\lambda}+R_{\lambda})c_{\beta}\lambda^{\beta}+\epsilon\right\}\geq 1-2\exp\left\{-\frac{m \lambda^2\epsilon^2}{2p^2\omega_{\lambda}^2}\right\}
\end{align*}
Hence proved.
\end{proof}	
\vspace{0.5cm}
\noindent Now we go on to prove the next theorem for the case $\Omega(u)=\|u\|_{B_1}$.

\begin{proof}[\textbf{Proof of theorem \ref{thm2}:}]
	From the previous theorem, we know that
	\begin{align*}
		\underset{z\in Z^m}{Prob}\left\{D_{\Omega}^{\xi(u_{z}^{\lambda})}(u_{\rho},u_{z}^{\lambda})\leq(M_{\lambda}+R_{\lambda})c_{\beta}\lambda^{\beta}+\epsilon\right\}\geq 1-2\exp\left\{-\frac{m \lambda^2\epsilon^2}{2p^2\omega_{\lambda}^2}\right\}
	\end{align*}
	We substitute $\delta$ for $2\exp\left\{-\frac{m \lambda^2\epsilon^2}{2p^2\omega_{\lambda}^2}\right\}$, so $\epsilon$ can be written as $\frac{\sqrt{2\ln(2/\delta)}p\omega_{\lambda}}{\lambda\sqrt{m}}$. Then the above inequality will be
	\begin{align*}
		\underset{z\in Z^m}{Prob}\left\{D_{\Omega}^{\xi(u_{z}^{\lambda})}(u_{\rho},u_{z}^{\lambda})\leq(M_{\lambda}+R_{\lambda})c_{\beta}\lambda^{\beta}+\frac{\sqrt{2\ln(2/\delta)}p\omega_{\lambda}}{\lambda\sqrt{m}}\right\}\geq 1-\delta
	\end{align*}
	This proves the first part of the theorem. Now consider $\Omega(u)=\|u\|_{B_1}$. Note that even though $\Omega$ is not strictly convex, the strict convexity of $\mathcal{E}_{\rho}$ makes the sum $\mathcal{E}_{\rho}+\Omega$ strictly convex. Hence, the minimizers of \ref{reg} and \ref{regtrue} are well defined and exist uniquely.
	We need to find $R_\lambda$ and $M_{\lambda}$ such that  $\|\tau(u_{\rho})-\zeta(u_{\rho}^{\lambda})\|_{B_1^*}\leq R_{\lambda}$ and $\|\zeta(u_{\rho}^{\lambda})-\xi(u_{z}^{\lambda})\|_{B_1^*}\leq M_{\lambda}$, where $\tau(u_{\rho})\in\partial\Omega(u_{\rho})$ and $\zeta(u_{\rho}^{\lambda})\in\partial\Omega(u_{\rho}^{\lambda})$. We note that
	\begin{equation*}
		\|\tau(u_{\rho})-\zeta(u_{\rho}^{\lambda})\|_{B_1^*}\leq \|\tau(u_{\rho})\|_{B_1^*}+\|\zeta(u_{\rho}^{\lambda})\|_{B_1^*}
	\end{equation*}
	and
	\begin{equation*}
		\|\zeta(u_{\rho}^{\lambda})-\xi(u_{z}^{\lambda})\|_{B_1^*}\leq \|\zeta(u_{\rho}^{\lambda})\|_{B_1^*}+\|\xi(u_{z}^{\lambda})\|_{B_1^*}
	\end{equation*}
	From \cite{schuster2012regularization}, theorem $2.28$, we realize that $\|\tau(u_{\rho})\|_{B_1^*}=\|\zeta(u_{\rho}^{\lambda})\|_{B_1^*}=\|\xi(u_{z}^{\lambda})\|_{B_1^*}=1$, and therefore $M_{\lambda}=R_{\lambda}=2$ is an appropriate choice. Then from \eqref{ineq22}, with a confidence of $1-\delta$, the following statement holds.
	\begin{equation}\label{ineq23}
		D_{\Omega}^{\xi(u_{z}^{\lambda})}(u_{\rho},u_{z}^{\lambda})\leq4c_{\beta}\lambda^{\beta}+\frac{\sqrt{2\ln(2/\delta)}p\omega_{\lambda}}{\lambda\sqrt{m}}
	\end{equation}
	Now we need to find a value for $\omega_{\lambda}$ such that
\begin{equation*}
2k(\|u_{\rho}^{\lambda}\|_{B_1}+\|u_{z}^{\lambda}\|_{B_1})[M+k\|u_z^{\lambda}\|_{B_1}]^{p-1}
\leq\omega_{\lambda}
\end{equation*}
Compare \eqref{true} and \eqref{regtrue}. Since $\mathcal{E}_{\rho}(u_{\rho})\leq\mathcal{E}_{\rho}(u_{\rho}^{\lambda})$ it is necessary that $\|u_{\rho}^{\lambda}\|_{B_1}\leq\|u_{\rho}\|_{B_1}$. Also, from \ref{reg}, we get
\begin{equation*}
\left(\frac{1}{m}\sum_{i=1}^m|y_i-(Au)(x_i)|^p+\lambda\|u_z^{\lambda}\|_{B_1}\right)\leq \left(\frac{1}{m}\sum_{i=1}^m|y_i-(A(0))(x_i)|^p+\lambda\|0\|_{B_1}\right)
\end{equation*}
which will produce the following bound
\begin{equation*}
	\lambda\|u_z^{\lambda}\|_{B_1}\leq \frac{1}{m}\sum_{i=1}^m|y_i|^p\leq M^p
\end{equation*}
Then, by using the fact that $\|u_{\rho}^{\lambda}\|_{B_1}\leq\|u_{\rho}\|_{B_1}$ and $\|u_z^{\lambda}\|_{B_1}\leq \frac{M^p}{\lambda}$, we can argue that
\begin{align*}
	2k(\|u_{\rho}^{\lambda}\|_{B_1}+\|u_{z}^{\lambda}\|_{B_1})[M+k\|u_z^{\lambda}\|_{B_1}]^{p-1}
	&\leq 2\left(k\|u_{\rho}\|_{B_1}+\frac{kM^p}{\lambda}\right)\left[M+\frac{kM^p}{\lambda}\right]^{p-1}\\
	&\leq 2\left[\sigma_{\max}+\frac{kM^p}{\lambda}\right]^{p}=\omega_{\lambda}
\end{align*}
where $\sigma_{\max}$ is used to denote $\max\{k\|u_{\rho}\|_{B_1},M\}$. Hence, this value qualifies as an upper bound and can be used for $\omega_{\lambda}$.
By \eqref{ineq23}, we say that with a confidence of $1-\delta$, the following inequality is true.
\begin{equation*}
		D_{\Omega}^{\xi(u_{z}^{\lambda})}(u_{\rho},u_{z}^{\lambda})\leq4c_{\beta}\lambda^{\beta}+{2p{\sqrt{2\ln(2/\delta)}}}~\frac{\left[\lambda\sigma_{\max}+kM^p\right]^{p}}{\lambda^{p+1}\sqrt{m}}
\end{equation*}
Hence proved.
\end{proof}
\vspace{0.3cm}
\noindent We will proceed to prove the corollary.
\begin{proof}[\textbf{Proof of corollary \ref{cor3}:}]
	We have with a confidence of $1-\delta$ that
	\begin{equation*}
		D_{\Omega}^{\xi(u_{z}^{\lambda})}(u_{\rho},u_{z}^{\lambda})\leq4c_{\beta}\lambda^{\beta}+{2p{\sqrt{2\ln(2/\delta)}}}~\frac{\left[\lambda\sigma_{\max}+kM^p\right]^{p}}{\lambda^{p+1}\sqrt{m}}
	\end{equation*}
	Substitute $\lambda=\frac{1}{m^{\alpha}}$ for some $\alpha>0$. This will give us
		\begin{equation*}
		D_{\Omega}^{\xi(u_{z}^{\lambda})}(u_{\rho},u_{z}^{\lambda})\leq\frac{4c_{\beta}}{m^{\alpha\beta}}+2p\sqrt{2\ln(2/\delta)}\frac{\left[\frac{\sigma_{\max}}{{m^{\alpha}}}+kM^p\right]^{p}m^{\alpha(p+1)}}{\sqrt{m}}
	\end{equation*}
	If $m$ is large enough, then the term $\left[\frac{\sigma_{\max}}{{m^{\alpha}}}+kM^p\right]^{p}\approx \left[kM^p\right]^{p}$ and we have
	
	\begin{equation*}
		D_{\Omega}^{\xi(u_{z}^{\lambda})}(u_{\rho},u_{z}^{\lambda})\leq\frac{4c_{\beta}}{m^{\alpha\beta}}+\frac{2p\sqrt{2\ln(2/\delta)}\left(kM^p\right)^{p}}{m^{1/2-\alpha p-\alpha}}
	\end{equation*}
	Now we use $\alpha=\frac{1}{2(\beta+p+1)}$ as given by \eqref{alp}, which, after simplification will yield the following inequality
			\begin{equation*}
		D_{\Omega}^{\xi(u_{z}^{\lambda})}(u_{\rho},u_{z}^{\lambda})\leq{4c_{\beta}}{m^{-\frac{\beta}{2(\beta+p+1)}}}+{2p\sqrt{2\ln(2/\delta)}\left(kM^p\right)^{p}}{m^{-\frac{\beta}{2(\beta+p+1)}}}
	\end{equation*}
	From which we can conclude with a confidence of $1-\delta$ that
	\begin{equation*}
		D_{\Omega}^{\xi(u_{z}^{\lambda})}(u_{\rho},u_{z}^{\lambda})=O(m^{-\frac{\beta}{2(\beta+p+1)}})= O(\lambda^{\beta})
	\end{equation*}
	Hence proved.
\end{proof}
\vspace{0.3cm}
\noindent The proof for the final theorem can now be established.
\begin{proof}[\textbf{Proof of theorem \ref{thm3}:}]
		From \cite{JOSEPH2026102050}, the minimizers $u_{\rho}^{\lambda}$ and $u_z^{\lambda}$ are well defined.
	Now in lemma \ref{lemma3}, by interchanging $\gamma$ and $\rho$, we get
	\begin{align*}
		\frac{\lambda}{p} D_{\Omega}^{\xi(u_{\rho}^{\lambda})}(u_{\gamma}^{\lambda},u_{\rho}^{\lambda}) \leq
		\Big\langle u_{\gamma}^{\lambda}-u_{\rho}^{\lambda}, \int_{Z}S_x^*(W_{\rho}^{\lambda}(x,y) ) d \gamma-\int_{Z}S_x^*(W_{\rho}^{\lambda}(x,y) ) d \rho \Big\rangle_{B_1\times
			B_1^*}
	\end{align*}
We combine with \eqref{qcon2} and use the empirical distribution $\gamma=\gamma_z$ to get
		\begin{align*}
		\frac{\lambda c_q}{p} \|u_{z}^{\lambda}-u_{\rho}^{\lambda}\|_{B_1}^q \leq
		\Big\langle u_{z}^{\lambda}-u_{\rho}^{\lambda}, \frac{1}{m}\sum_{i=1}^{m}S_{x_i}^*(W_{\rho}^{\lambda}(x_i,y_i) )-\int_{Z}S_x^*(W_{\rho}^{\lambda}(x,y) ) d \rho \Big\rangle_{B_1\times
			B_1^*}
	\end{align*}
	We can now conclude the following
\begin{align}\label{ineq24}
	\frac{\lambda c_q}{p} \|u_{z}^{\lambda}-u_{\rho}^{\lambda}\|_{B_1}^{q-1} \leq
	\left\| \int_{Z}S_x^*(W_{\rho}^{\lambda}(x,y) ) d \rho -\frac{1}{m}\sum_{i=1}^{m}S_{x_i}^*(W_{\rho}^{\lambda}(x_i,y_i) )\right\|_{
		B_1^*}
\end{align}
			Since the space $B_1$ is uniformly convex, and hence reflexive, there exists some $u\in B_1$ with $\|u\|_{B_1}\leq1$ such that
	\begin{align*}
		&\left\| \int_{Z}S_x^*(W_{\rho}^{\lambda}(x,y) ) d \rho -\frac{1}{m}\sum_{i=1}^{m}S_{x_i}^*(W_{\rho}^{\lambda}(x_i,y_i) )\right\|_{
		B_1^*}\\
		& = \Big| \Big\langle u, \int_{Z}S_x^*(W_{\rho}^{\lambda}(x,y) ) d \rho -\frac{1}{m}\sum_{i=1}^{m}S_{x_i}^*(W_{\rho}^{\lambda}(x_i,y_i) ) \Big\rangle_{B_1\times B_1^*}
		\Big|\\
				&= \Big| \int_{Z} W_{\rho}^{\lambda}(x,y) (Au)(x)\, d\rho - \frac{1}{m} \sum_{i=1}^{m} W_{\rho}^{\lambda}(x_i,y_i) (Au)(x_i)
				\Big|\\
				& = \left|\int_{Z} h(z) d \rho-\frac{1}{m} \sum_{i=1}^{m}
				h\left(z_{i}\right)\right|
			\end{align*}
			where we used the notation $h(z)=h(x,y):=W_{\rho}^{\lambda}(x,y)(Au)(x)$.\\
			
			\noindent	Thus, from \eqref{ineq24} we obtain
			\begin{equation*}
				\frac{\lambda c_q}{p}\left\|u_{z}^{\lambda}-u_{\rho}^{\lambda}\right\|_{B_1}^{q-1} \leq\left|\int_{Z} h(z) d \rho-\frac{1}{m} \sum_{i=1}^{m}
				h\left(z_{i}\right)\right|
			\end{equation*}

			\vspace{0.2cm}
			\noindent Using the same argument as before, we can show that
			\begin{align}\label{ineq25}
				\underset{z\in Z^m}{\text{Prob}}\left\{\frac{\lambda c_{q}}{p}\left\|u_{z}^{\lambda}-u_{\rho}^{\lambda}\right\|_{{B_1}}^{q-1} \leq \epsilon\right\}
				&\geq \underset{z\in Z^m}{\text{Prob}}\left\{\left|\int_{Z} h(z) d \rho-\frac{1}{m} \sum_{i=1}^{m}
				h\left(z_{i}\right)\right|\leq
				\epsilon\right\}
			\end{align}
			
		\noindent Now, to use \eqref{cor}, we first show that $|h(z)|$ is bounded. We see that
			\begin{align*}
				|h(z)|&\leq(|y|+\left|(Au_{\rho}^{\lambda})(x)|\right)^{p-1}k\n{u}_{B_1}\\
				&\leq(M+k\|u_{\rho}^{\lambda}\|_{B_1})^{p-1}k\cdot1\\
				&\leq(M+k\|u_{\rho}\|_{B_1})^{p-1}k.
			\end{align*}
			Again, since $|h(z)-\int_Z h(z)d\rho|\leq |h(z)|+\int_Z |h(z)|d\rho\leq2(M+k\|u_{\rho}\|_{B_1})^{p-1}k$, hence by lemma \ref{lemma4} we get
			\begin{equation*}
				\underset{z\in Z^m}{\text{Prob}}\left\{\left|\int_{Z} h(z) d \rho-\frac{1}{m} \sum_{i=1}^{m}
				h\left(z_{i}\right)\right|\leq
				\epsilon\right\}\geq 1-2\exp\Big\{-\frac{m\epsilon^2}{2\omega^2}\Big\}
			\end{equation*}
	where $\omega=2(M+k\|u_{\rho}\|_{B_1})^{p-1}k$, a constant independent of $\lambda$. Hence, combining with \eqref{ineq25}, we claim
\begin{equation*}
	\underset{z\in Z^m}{\text{Prob}}\left\{\left\|u_z^{\lambda}-u_{\rho}^{\lambda}\right\|_{{B_1}} \leq
	\left(\frac{p\epsilon}{\lambda c_q}\right)^{\frac{1}{q-1}}\right\} \geq 1-2 \exp \left(-\frac{m \epsilon^{2}}{2
		\omega^{2}}\right)
\end{equation*}
This is analogous to
\begin{align}\label{ineq26}
	\underset{z\in Z^m}{\text{Prob}}\left\{\left\|u_z^{\lambda}-u_{\rho}^{\lambda}\right\|_{B_1} \leq \epsilon\right\}~\geq 1- 2 \exp \left(-\frac{m \lambda^{2} c_{q}^{2} \epsilon^{2(q-1)}}{2p^2
		\omega^{2}}\right)
\end{align}
		Now we use triangle inequality to establish the final part of the proof. From assumption $2$, we have the following.
		\begin{align*}
		\n{u_z^{\lambda}-\p u}_{B_1}&\leq\n{u_z^{\lambda}-u_{\rho}^{\lambda}}_{B_1}+\n{u_{\rho}^{\lambda}-\p u}_{B_1}\\
		\Rightarrow\n{u_z^{\lambda}-\p
			u}_{B_1}&\leq\n{u_z^{\lambda}-u_{\rho}^{\lambda}}_{B_1}+c_{\beta}{\lambda}^{\beta}
	\end{align*}
	This gives us
	\begin{equation*}
		\left(\left\|u_z^{\lambda}-u_{\rho}\right\|_{B_1}-c_{\beta}{\lambda}^{\beta}\right)
		\leq\left\|u_z^{\lambda}-u_{\rho}^{\lambda}\right\|_{B_1}
	\end{equation*}
From the above expression, we conclude, as we did earlier, that for any $\epsilon>0$
	\begin{align}
		\underset{z\in Z^m}{\text{Prob}}\left\{\left\|u_z^{\lambda}-u_{\rho}\right\|_{B_1}-c_{\beta}{\lambda}^{\beta}\leq\epsilon\right\} \geq     \underset{z\in Z^m}{\text{Prob}}\left\{\left\|u_z^{\lambda}-u_{\rho}^{\lambda}\right\|_{B_1}\leq\epsilon\right\}
	\end{align}
	Then by \eqref{ineq26}, we prove the statement of the theorem.		
	\begin{align*}
	\underset{z\in Z^m}{\text{Prob}}\left\{\left\|u_z^{\lambda}-u_{\rho}\right\|_{B_1}\leq c_{\beta}{\lambda}^{\beta}+\epsilon\right\} 
		\geq 1-2 \exp \left(-\frac{m \lambda^{2} c_{q}^{2} \epsilon^{2(q-1)}}{2p^2
			\omega^{2}}\right)
	\end{align*}
	
\end{proof}
\vspace{0.3cm}
\begin{proof}[\textbf{Proof of corollary \ref{cor2}:}]
	From the statement of the theorem \ref{thm3}, take
	\begin{equation*}
		2 \exp \left(-\frac{m \lambda^{2} c_{q}^{2} \epsilon^{2(q-1)}}{2p^2
			\omega^{2}}\right)=\delta,
	\end{equation*}
	Then $\epsilon$ will have the form given by
	\begin{equation*}
		\epsilon=\left(\frac{2p^2\omega^2\ln(2/\delta)}{c_q^2m\lambda^2}\right)^{\frac{1}{2(q-1)}}
	\end{equation*}
	Then \eqref{ineq27} becomes
	
		\begin{equation*}
		\underset{z\in Z^m}{\text{Prob}}  \left\{\left\|u_z^{\lambda}-u_{\rho}\right\|_{B_1} \leq  \sqrt[2(q-1)]{\frac{2p^2\omega^2\ln(2/\delta)}{c_q^2}}\left(\frac{1}{\lambda\sqrt{m}}\right)^{\frac{1}{(q-1)}}+c_{\beta}{\lambda}^{\beta}\right\} \geq 1-\delta
	\end{equation*}
	
	\noindent For simplicity, let's use the following notation
	
	\begin{equation*}
	\sqrt[2(q-1)]{\frac{2p^2\omega^2\ln(2/\delta)}{c_q^2}}=a_{\delta}
	\end{equation*}
	In addition, replace $\lambda=m^{-\frac{1}{2(1+\beta(q-1))}}$ as given and simplify the terms, we then have
	
\begin{equation*}
		\underset{z\in Z^m}{\text{Prob}}  \left\{\left\|u_z^{\lambda}-u_{\rho}\right\|_{B_1} \leq  	a_{\delta}{m}^{-\frac{\beta}{2(1+\beta(q-1))}}+c_{\beta}{{m^{-\frac{\beta}{2(1+\beta(q-1))}}}}\right\} \geq 1-\delta
	\end{equation*}
	That is,
	
	\begin{equation*}
		\underset{z\in Z^m}{\text{Prob}}  \left\{\left\|u_z^{\lambda}-u_{\rho}\right\|_{B_1} \leq  	(a_{\delta}+c_{\beta}){m}^{-\frac{\beta}{2(1+\beta(q-1))}}\right\} \geq 1-\delta
	\end{equation*}
That is for any fixed $\delta\in(0,1)$, we can say with a confidence of at least $1-\delta$ that

 	\begin{equation*}
  \left\|u_z^{\lambda}-u_{\rho}\right\|_{B_1} = O\left( 	{m}^{-\frac{\beta}{2(1+\beta(q-1))}}\right)
 \end{equation*}
	Hence the proof.
\end{proof}

\section{Numerical illustration}\label{sec5}
In this section we provide an application of theorem \ref{thm2}. We consider the input space $B_1$ to be $l^1$, the set of absolutely summable real sequences. Let for all
$i\in\mathbb{N}$, $\phi_i\in C[-1,1]$ denote the Legendre polynomial of degree $i-1$. We assume that each $\phi_i$ is standardized such that $\|\phi_i\|_{\infty}\leq1$, for all $i$. Consider the sequence $(\alpha_1,\alpha_2,...)$ that converges to $0$. We then define the operator $A:l^1\to C[-1,1]\subset L^2[-1,1]$ as
\begin{equation}
	Au=\sum_{i=1}^{\infty}\alpha_i u_i\phi_{i}
\end{equation}
where $u=(u_1,u_2,...)$ is an element of $l^1$. The above equation is well defined since for each $u$, $$(Au)(x)\leq\left|\sum_{i=1}^{\infty}\alpha_i u_i\phi_{i}(x)\right|\leq \max\{|\alpha_{i}|\}\sum_{i=1}^{\infty}|u_i|=k\|u\|_{l^1}<\infty.$$
Thus, assumption A1 is satisfied with $k=\max\{|\alpha_{i}|\}$. 
To show that the image lies in $C[-1,1]$, consider 
 $$\left|(Au)(x)-\sum_{i=1}^{N}\alpha_i u_i\phi_i(x)\right|=\left|\sum_{i=N+1}^{\infty}\alpha_i u_i\phi_i(x)\right|\leq \max\{|\alpha_{i}|\}\sum_{i=N+1}^{\infty}\left|u_i\right|$$
 The last term goes to zero independent of the point $x$. This implies that each $Au$ is the uniform limit of continuous functions; hence, $Au$ is continuous for all $u\in l^1$ as required.\\

 Consider $c_0$, the set of real sequences that converge to zero. Recall that $c_0^*$ is isometric to $l^1$. Let $W$ denote the weak* topology on $l^1$ induced by $c_0$, that is, $W=\sigma(l^1,c_0)$. 
 
 To prove that $S_x$ is weak* lower semi-continuous for all $x\in [-1,1]$, it is enough to show that it is weak* continuous. That is for any sequence $\{u_n\}_{n\in\mathbb{N}}\subset l^1$ such that $u_n\overset{*}{\rightharpoonup}u$ for some $u\in l^1$, it must imply that $S_x(u_n)\rightarrow S_x(u)$. Let $u_n$ be such that $u_n\overset{*}{\rightharpoonup}u$. Then by the definition of weak* convergence it must follow that $\langle b,u_n\rangle_{c_0\times l^1}\to \langle b,u\rangle_{c_0\times l^1}$ for all $b=(b_1,b_2,...)\in c_0$, i.e., $$\sum_{i=1}^{\infty}u_{ni}b_i\rightarrow \sum_{i=1}^{\infty}u_{i}b_i,$$
 
 Since $\{\alpha_{i}\phi_{i}(x)\}_{i\in\mathbb{N}}\in c_0$ for all $x\in[-1,1]$, it indicates that 
 
 $$S_x(u_n)=\sum_{i=1}^{\infty}u_{ni}\alpha_i\phi_i(x)\rightarrow \sum_{i=1}^{\infty}u_{i}\alpha_i\phi_i(x)=S_x(u).$$
 
 Therefore the operator $S_x$ is weak* continuous, and hence weak* lower semi-continuous. Moreover, $\Omega(u)=\|u\|_{l^1}$ is also weak* lower semi-continuous (\cite{Conway1990}). Further, by the Banach–Alaoglu theorem, the unit ball is weak* compact, satisfying $(B3)$. Therefore, all the assumptions in \ref{assume1}, except A2, are shown to be satisfied.\\
 
  \noindent We take $\alpha_i=i^{-r}, ~r>0$. Then the operator $A$ is given by
  \begin{equation}
  	Au=\sum_{i=1}^{\infty}i^{-r}u_i\phi_{i}
  \end{equation}
 We simulate the data points as follows. Let $u_{\rho}$ be an element of $l^1$. Consider the data points
 $z=\{(x_i,y_i)\}_{i=1}^{m}$ that follow the probability distribution $\rho(x,y)$.
 Our distribution $\rho(x,y)$ is given such that the marginal distribution $\rho_X(x)$ is the uniform distribution on $[-1,1]$ and the conditional distribution $\rho(y|x)$ is
 \begin{equation*}
 	\rho(y|x)=(Au_{\rho})(x)+ N(0,\sigma^2),
 \end{equation*}
 where $N(0,\sigma^2)$ is the Guassian distribution with mean $0$ and variance $\sigma^2$. Then,
 it follows from \cite{CuckerZhou2007} that for $p=2$ in \eqref{true} we have,
 \begin{equation*}
 	u_{\rho}=\underset{u\in 1^1}{\arg \min}\int_Z (y-(Au)(x))^2 d\rho(x,y).
 \end{equation*}
 Recall that the element $u_{\rho}^{\lambda}$ is defined as follows. For all $u\in l^1$

 \begin{equation*}
 	\mathcal{E_{\rho}}(u_{\rho}^{\lambda})-	\mathcal{E_{\rho}}(u_{\rho})+\lambda\|u_{\rho}^{\lambda}\|_{B_1}\leq \mathcal{E_{\rho}}(u)-	\mathcal{E_{\rho}}(u_{\rho})+\lambda\|u\|_{B_1}
 \end{equation*}
By \cite{CuckerZhou2007} we know that $\mathcal{E_{\rho}}(u)-	\mathcal{E_{\rho}}(u_{\rho})=	\|Au-Au_{\rho}\|_{L^2}^2$. Hence, from the above inequality we deduce that
\begin{equation*}
	u_{\rho}^{\lambda}=\underset{u\in l^1}{\arg\min}\left(	\|Au-Au_{\rho}\|_{L^2}^2+\lambda\|u\|_{B_1}\right)
\end{equation*}
The Legendre polynomials $\phi_i$ form an orthonormal basis for $L^2[-1,1]$. Therefore we have 
\begin{equation*}
	u_{\rho}^{\lambda}=\underset{u\in l^1}{\arg\min}\left(	\sum_{i=1}^{\infty}(u_{i}-u_{\rho,i})^2i^{-2r}+\lambda\sum_{i=1}^{\infty}|u_i|\right)
\end{equation*}
Since each term in the summation is positive and independent, each of them should minimize the corresponding coordinates independently. Hence $u_{\rho,i}^{\lambda}$ is the minimizer of $(u_{i}-u_{\rho,i})^2i^{-2r}+\lambda|u_i|$. This holds only if
\begin{equation}\label{u}
	u_{\rho,i}^{\lambda}=\text{sgn}(u_{\rho,i})\max\left(|u_{\rho,i}|-\frac{\lambda i^{2r}}{2},0\right)
\end{equation}
 Let's take $N_{\lambda}$ to be the largest integer $i$ such that $|u_{\rho,i}|-{(\lambda/2) i^{2r}}\geq0$. If we assume $u_{\rho}$ is such that $u_{\rho,i}\leq bi^{-lv},b>0,lv>1$, this implies that $N_{\lambda}\leq (\lambda/c)^{-\frac{1}{2r+lv}}$ for some constant $c>0$. Then from \eqref{u}, we get
 \begin{align*}
 	u_{\rho,i}^{\lambda}&=u_{\rho,i}-\text{sgn}(u_{\rho,i})\frac{\lambda i^{2r}}{2},~~~\text{when}~i\leq N_{\lambda}\\
 	u_{\rho,i}^{\lambda}&=0,\quad~~~~~~~~~~~~~~~~~~~~~~~~~~~~~\text{when}~i> N_{\lambda}
 \end{align*}
  We then have the following inequality  for all $i\leq N_{\lambda}$
  \begin{equation*}
  	|u_{\rho,i}^{\lambda}-u_{\rho,i}|=\frac{\lambda i^{2r}}{2}
  \end{equation*}
 Similarly, for $i>N_{\lambda}$,
 \begin{equation*}
 	|u_{\rho,i}^{\lambda}-u_{\rho,i}|=|u_{\rho,i}|
 \end{equation*}
The $l^1$ norm of $u_{\rho}^{\lambda}-u_{\rho}$ is given by
\begin{align*}
	\|u_{\rho}^{\lambda}-u_{\rho}\|_{l^1}=&\sum_{i=1}^{N_{\lambda}}|u_{\rho,i}^{\lambda}-u_{\rho,i}|+|u_{\rho,N_{\lambda}+1}^{\lambda}-u_{\rho,N_{\lambda}+1}|+\sum_{i=N_{\lambda}+2}^{\infty}|u_{\rho,i}^{\lambda}-u_{\rho,i}|\\
	=&\sum_{i=1}^{N_{\lambda}}\frac{\lambda i^{2r}}{2}+u_{\rho,N_{\lambda}+1}+\sum_{i=N_{\lambda}+2}^{\infty}u_{\rho,i}\\
	\leq&\sum_{i=1}^{N_{\lambda}}\frac{\lambda {N_{\lambda}}^{2r}}{2}+b({N_{\lambda}+1})^{-lv}+b\int_{{N_{\lambda}+1}}^{\infty}x^{-lv}dx
\end{align*}
Hence, the LHS satisfies the inequality
\begin{equation*}
		\|u_{\rho}^{\lambda}-u_{\rho}\|_{l^1}\leq \frac{\lambda {N_{\lambda}}^{2r+1}}{2}+b({N_{\lambda}+1})^{-lv}+\frac{{b(N_{\lambda}+1)}^{-lv+1}}{lv-1}
\end{equation*}
By the definition of $N_{\lambda}$, we had $N_{\lambda}\leq (\lambda/c)^{-\frac{1}{2r+lv}}$ and $N_{\lambda}+1\geq (\lambda/c)^{-\frac{1}{2r+lv}}$. We apply this inequality to the above expression. Since $-lv+1<0$, we end up with
\begin{equation*}
	\|u_{\rho}^{\lambda}-u_{\rho}\|_{l^1}\leq  k_1{\lambda}^{\frac{lv-1}{2r+lv}}+k_2\lambda^{\frac{lv}{2r+lv}}+k_3\lambda^{\frac{lv-1}{2r+lv}}
\end{equation*}
where $k_1,~k_2,~k_3$ are appropriate positive constants that represent the coefficients.

\noindent As we want the regularization parameter $\lambda$ to go to $0$, without loss of generality, assume that $0<\lambda<1$. Hence we have
\begin{equation*}
\|u_{\rho}^{\lambda}-u_{\rho}\|_{l^1}  \leq  k_1{\lambda}^{\frac{lv-1}{2r+lv}}+k_2\lambda^{\frac{lv-1}{2r+lv}}+k_3\lambda^{\frac{lv-1}{2r+lv}}=k{\lambda}^{\frac{lv-1}{2r+lv}}
\end{equation*}
Where $k=k_1+k_2+k_3$. Therefore assumption (A2) is satisfied with $\beta=\frac{lv-1}{2r+lv}$.\\

\noindent We take $u_{\rho,i}=\{i^{-2}\}_{i\in\mathbb{N}}$ and $\alpha_{i}=i^{-1}$. Further we consider $\sigma=0.1,0.05$ and $0.01$ for analysis purposes. In the given fig \ref{fig1}, the regularized solution $u_z^{\lambda}$ is plotted against $m=100,1000,10000$ against the noise level of $\sigma=0.1$.\\

\begin{figure}[H]
	\captionsetup{font=footnotesize}
	\centering
	\subfloat{\includegraphics[width
		= 4.5in, height=2.5in]{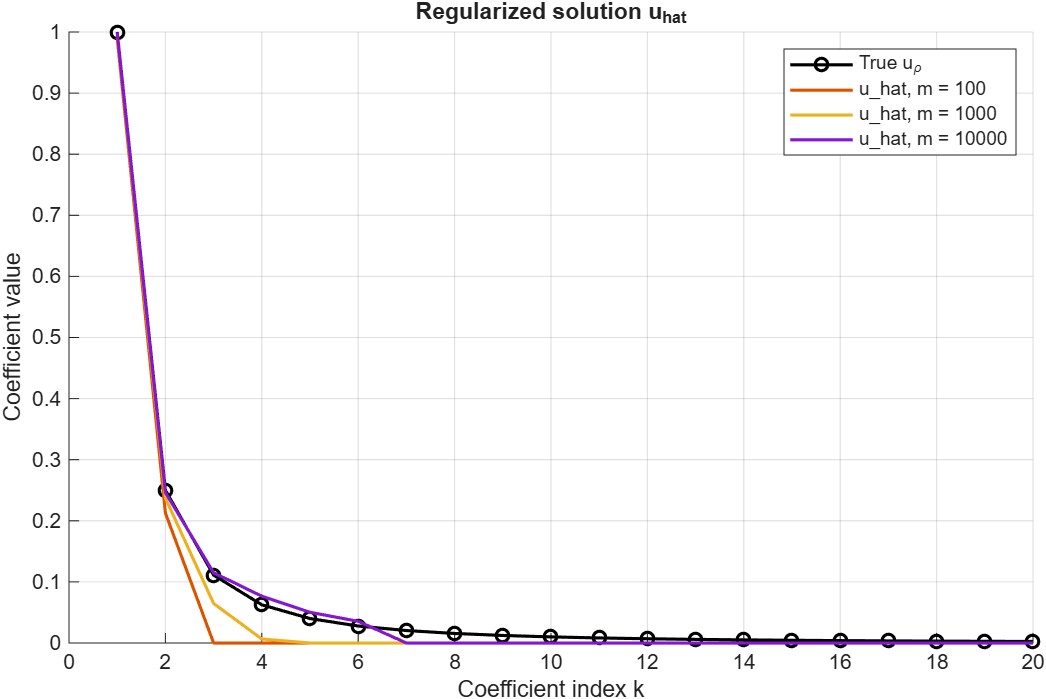}}
	\caption{Graph of $u_z^{\lambda}$ for $\sigma=0.1$ against $m=100,~1000,~10000$. } \label{fig1}
\end{figure} 

The given image below, fig \ref{fig2} depicts the regularized solution corresponding to $m=100,1000,10000$ when the noise level $\sigma=0.05$.
\begin{figure}[H]
	\captionsetup{font=footnotesize}
	\centering
	\subfloat{\includegraphics[width
		= 4.5in, height=2.5in]{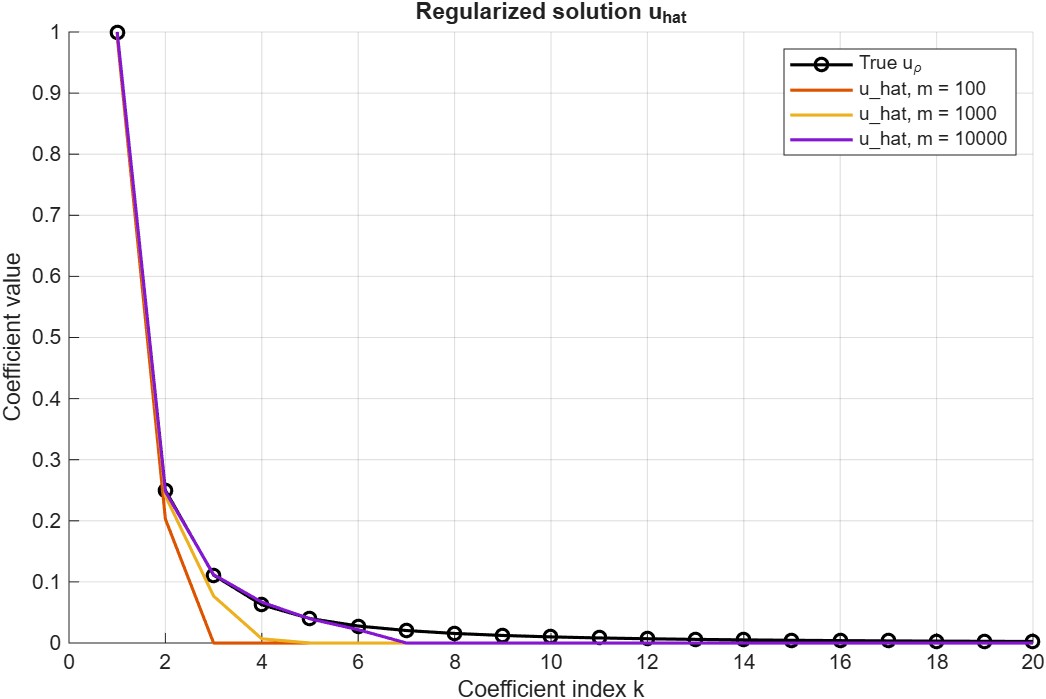}}
	\caption{Graph of $u_z^{\lambda}$ for $\sigma=0.05$ against $m=100,~1000,~10000$ } \label{fig2}
\end{figure}
The following fig \ref{fig3} shows the regularized solution corresponding to $m=100,1000,10000$ when the noise level $\sigma=0.01$.
\begin{figure}[H]
		\captionsetup{font=footnotesize}
\centering
\subfloat{\includegraphics[width
	= 4.5in, height=2.5in]{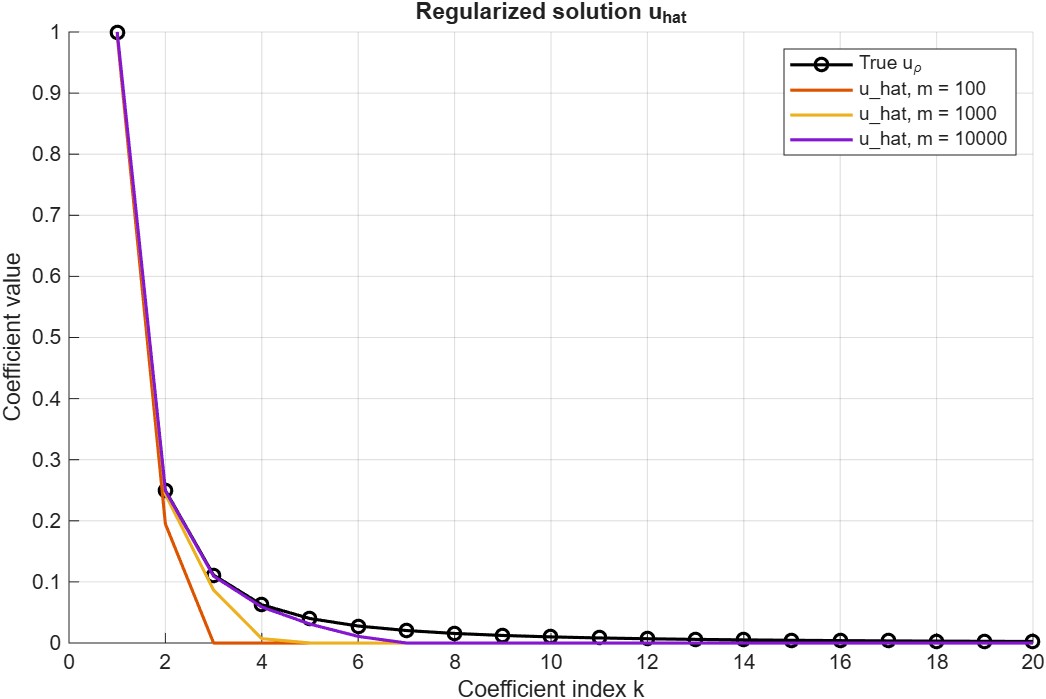}}
\caption{Graph of $u_z^{\lambda}$ for $\sigma=0.01$ against $m=100,~1000,~10000$ } \label{fig3}
\end{figure}
The Bregman distance corresponding to different noise levels and $m$ is tabulated as follows. The Bregman error can be seen converging to zero as the number of data points increases. We have not used corollary \ref{cor3} since the rate obtained is not sharp.
\begin{table}[H]
	\caption{Computational results for different error levels }
	\begin{center} 
		\begin{tabular}{|c|c|c|c|}
			\hline $\sigma$ & $m$  & Regul.Parameter, $\lambda=m^{-1}$
			& \multicolumn{1}{|c|}{Breman error}\\
			\hline
			& 100 & $10^{-2}$  & $0.28652$ \\
			\cline{2-4}
			$10\%$  & 1000 & $10^{-3}$  & $0.19968$ \\
			\cline{2-4}
			& 10000& $10^{-4}$  & $0.16825$\\
			\hline
			& 100 & $10^{-2}$  & $0.30196$ \\
			\cline{2-4}
			$5\%$  & 1000 & $10^{-3}$  & $0.19722$ \\
			\cline{2-4}
			& 10000& $10^{-4}$  & $0.14693$\\
			\hline
			& 100 & $10^{-2}$  & $0.3143$ \\
			\cline{2-4}
			$1\%$  & 1000 & $10^{-3}$  & $0.19525$ \\
			\cline{2-4}
			& 10000& $10^{-4}$  & $0.12988$\\
			\hline
		\end{tabular}\label{table1}
	\end{center}
\end{table}

\section{Conclusion and future work}\label{sec6}
In this paper we studied inverse learning problem in the Banach space setting without assuming reflexivity. We have achieved an upper rate of convergence for the Bregman distance in terms of the number of data points $m$. Furthermore, numerical simulations were
performed to demonstrate the implementability and effectiveness of
the proposed method. When $\Omega(u)=\|u\|_{B_1}$, we achieve an asymptotic rate of $m^{-\frac{\beta}{2(p+1+\beta)}}$. When our domain is a $q-$uniform convex Banach space, we get an improved rate of $m^{-\frac{\beta}{2(1+\beta(q-1))}}$ in terms of the norm convergence, compared to $m^{-\frac{\beta}{(2+s)(1+\beta(q-1))}},s>0$ from \cite{JOSEPH2026102050}. In particular, when $B_1$ is a Hilbert space, $q$ will be equal to $2$, and hence we get an upper rate of convergence of $m^{-\frac{\beta}{2(1+\beta)}}$.
 
 In \cite{BGM}, the optimal rate of convergence in the case of a Reproducing Kernel Hilbert space is of the order
\begin{equation*}
	\|L_v^l(u_z-u_{\rho})\|_{H_1}=O\left(\left(\frac{1}{m}\right)^{\frac{b(r+l)}{2br+b+1}}\right)
\end{equation*}
Here $u_z$ represents the regularized solution and $L_v$ is a self adjoint operator. The value $l$ belongs to $[0,1/2]$ and $b>1$ is a parameter corresponding to $\rho$ and $L_v$. The parameter $r$ appears in the source condition as $f_{\rho}=L_v^r(h)$, for some $h$ in the domain.

\noindent When $l=0$, the result is comparable to ours and provides a rate of
\begin{equation*}
	\|u_z-u_{\rho}\|_{H_1}=O\left(\left(\frac{1}{m}\right)^{\frac{br}{2br+b+1}}\right)
\end{equation*}
Our rate coincides with the optimal convergence rate for the special case when
\begin{equation}\label{comp}
	{\frac{\beta}{2(1+\beta)}}=\frac{br}{2br+b+1}
\end{equation}
As it can be seen, our result can be made optimal in the Hilbert setting for the cases when the powers coincide.\\
 We emphasize that the upper convergence rate established in the paper is not sharp in general, and hence does not imply an optimal rate. The issue of optimality on a larger class of domains is still open and is currently being investigated by us. We anticipate establishing improved, possibly optimal rates in a subsequent study. 
 
The proposed framework can be further generalized by
considering a non linear function $F:B_1\to V$ instead of a linear operator $A$. This allows us to consider more general problems arising in the inverse problem literature. Such generalized operators can be seen in \cite{zeidler1985nonlinear}, \cite{schuster2012regularization}.

\bibliographystyle{plain}
\bibliography{2ref}

\begin{thebibliography}{10}

\bibitem{doi:10.1137/24M1719785}
Giovanni~S. Alberti, Ernesto De~Vito, Tapio Helin, Matti Lassas, Luca Ratti,
  and Matteo Santacesaria.
\newblock Learning sparsity-promoting regularizers for linear inverse problems.
\newblock {\em SIAM Journal on Mathematics of Data Science}, 8(1):167--199,
  2026.

\bibitem{Alberti2026}
Giovanni~S. Alberti, Ernesto~De Vito, Tapio Helin, Matti Lassas, Luca Ratti,
  and Matteo Santacesaria.
\newblock Learning sparsity-promoting regularizers for linear inverse problems.
\newblock {\em SIAM Journal on Mathematics of Data Science}, 8(1):167--199,
  2026.

\bibitem{BAUER200752}
Frank Bauer, Sergei Pereverzev, and Lorenzo Rosasco.
\newblock On regularization algorithms in learning theory.
\newblock {\em Journal of Complexity}, 23(1):52--72, 2007.

\bibitem{BenesovaKruzik2017}
Barbora Bene{\v{s}}ov{\'a} and Martin Kru{\v{z}}{\'\i}k.
\newblock Weak lower semicontinuity of integral functionals and applications.
\newblock {\em SIAM Review}, 59(4):703--766, 2017.

\bibitem{BGM}
Gilles Blanchard and Nicole Mücke.
\newblock Optimal rates for regularization of statistical inverse learning
  problems.
\newblock {\em Foundations of Computational Mathematics}, 18, 04 2016.

\bibitem{Bregman1967}
L.~M. Bregman.
\newblock The relaxation method of finding the common point of convex sets and
  its application to the solution of problems in convex programming.
\newblock {\em USSR Computational Mathematics and Mathematical Physics},
  7(3):200--217, 1967.

\bibitem{Bubba2025O}
Tatiana~A. Bubba, Tommi Heikkil{\"a}, Demetrio Labate, and Luca Ratti.
\newblock Regularization with optimal space-time priors.
\newblock {\em SIAM Journal on Imaging Sciences}, 18(3):1563--1600, 2025.

\bibitem{Conway1990}
John~B. Conway.
\newblock {\em A Course in Functional Analysis}, volume~96 of {\em Graduate
  Texts in Mathematics}.
\newblock Springer-Verlag, New York, 2 edition, 1990.

\bibitem{CuckerZhou2007}
Felipe Cucker and Ding‐Xuan Zhou.
\newblock {\em Learning Theory: An Approximation Theory Viewpoint}.
\newblock Cambridge Monographs on Applied and Computational Mathematics, 24.
  Cambridge University Press, Cambridge, UK, 2007.

\bibitem{MR463994}
Ivar Ekeland and Roger Temam.
\newblock {\em Convex analysis and variational problems}, volume Vol. 1 of {\em
  Studies in Mathematics and its Applications}.
\newblock North-Holland Publishing Co., Amsterdam-Oxford; American Elsevier
  Publishing Co., Inc., New York, 1976.
\newblock Translated from the French.

\bibitem{engl2000regularization}
H.W. Engl, M.~Hanke, and A.~Neubauer.
\newblock {\em Regularization of Inverse Problems}.
\newblock Mathematics and Its Applications. Springer Netherlands, 2000.

\bibitem{Xu1991Inequalities}
K.~Xu H.\.
\newblock Inequalities in banach spaces with applications.
\newblock {\em Nonlinear Analysis}, 16(12):1127--1138, 1991.

\bibitem{Hanner1956}
Olof Hanner.
\newblock On the uniform convexity of $l^p$ and $\ell^p$.
\newblock {\em Arkiv för Matematik}, 1(5):237--244, 1956.

\bibitem{Hofmann2007}
Bernd Hofmann, Barbara Kaltenbacher, Christiane P{\"o}schl, and Otmar Scherzer.
\newblock A convergence rates result for tikhonov regularization in banach
  spaces with non-smooth operators.
\newblock {\em Inverse Problems}, 23(3):987--1010, 2007.

\bibitem{JOSEPH2026102050}
Darrel~K. Joseph and M.P. Rajan.
\newblock Convergence analysis of statistical inverse problems on reproducing
  kernel banach spaces.
\newblock {\em Journal of Complexity}, 95:102050, 2026.

\bibitem{Liu2017}
Huanxiang Liu, Baohuai Sheng, and Peixin Ye.
\newblock The improved learning rate for regularized regression with rkbss.
\newblock {\em International Journal of Machine Learning and Cybernetics},
  8(4):1235--1245, 2017.

\bibitem{Lorenz2008}
Dirk~A. Lorenz.
\newblock Convergence rates and source conditions for tikhonov regularization
  with sparsity constraints.
\newblock {\em Journal of Inverse and Ill-Posed Problems}, 16(5):463--478,
  2008.

\bibitem{nair2009linear}
M.T. Nair.
\newblock {\em Linear Operator Equations: Approximation and Regularization}.
\newblock World Scientific, 2009.

\bibitem{rockafellar1970convex}
R.~Tyrrell Rockafellar.
\newblock {\em Convex Analysis}.
\newblock Princeton University Press, Princeton, NJ, 1970.

\bibitem{schuster2012regularization}
T.~Schuster, B.~Kaltenbacher, B.~Hofmann, and K.S. Kazimierski.
\newblock {\em Regularization Methods in Banach Spaces}.
\newblock Radon Series on Computational and Applied Mathematics. De Gruyter,
  2012.

\bibitem{Learningrate}
Baohuai Sheng and Peixin Ye.
\newblock The learning rates of regularized regression based on reproducing
  kernel banach spaces.
\newblock {\em Abstract and Applied Analysis}, 2013(1):694181, 2013.

\bibitem{Showalter1997}
Ralph~E. Showalter.
\newblock {\em Monotone Operators in Banach Space and Nonlinear Partial
  Differential Equations}, volume~49 of {\em Mathematical Surveys and
  Monographs}.
\newblock American Mathematical Society, Providence, RI, 1997.

\bibitem{Tikhonov:1963}
A.~N. Tikhonov.
\newblock Solution of incorrectly formulated problems and the regularization
  method.
\newblock {\em Soviet Math. Dokl.}, 4(4):1035--1038, 1963.

\bibitem{JMLR:v22:20-751}
Rui Wang and Yuesheng Xu.
\newblock Representer theorems in banach spaces: Minimum norm interpolation,
  regularized learning and semi-discrete inverse problems.
\newblock {\em Journal of Machine Learning Research}, 22(225):1--65, 2021.

\bibitem{MR4749129}
Rui Wang, Yuesheng Xu, and Mingsong Yan.
\newblock Sparse representer theorems for learning in reproducing kernel
  {B}anach spaces.
\newblock {\em J. Mach. Learn. Res.}, 25:Paper No. [93], 45, 2024.

\bibitem{Xu2023SparseBanach}
Yuesheng Xu.
\newblock Sparse machine learning in banach spaces.
\newblock {\em Applied Numerical Mathematics}, 187:185--204, 2023.

\bibitem{XuRoach1991}
Z.~B. Xu and G.~F. Roach.
\newblock Characteristic inequalities of uniformly convex and uniformly smooth
  banach spaces.
\newblock {\em Journal of Mathematical Analysis and Applications},
  157(1):189--210, 1991.

\bibitem{zeidler1985nonlinear}
Eberhard Zeidler.
\newblock {\em Nonlinear Functional Analysis and its Applications, Volume III:
  Variational Methods and Optimization}.
\newblock Springer, New York, 1985.

\bibitem{5179093}
Haizhang Zhang, Yuesheng Xu, and Jun Zhang.
\newblock Reproducing kernel banach spaces for machine learning.
\newblock In {\em 2009 International Joint Conference on Neural Networks},
  pages 3520--3527, 2009.

\end{thebibliography}
\nocite{*}

\end{document}